\documentclass[a4paper,12pt]{article}

\usepackage[left=2cm,right=2cm, top=2cm,bottom=3cm,bindingoffset=0cm]{geometry}

\usepackage{verbatim}
\usepackage{amsmath}
\usepackage{amsthm}
\usepackage{amssymb}
\usepackage{delarray}
\usepackage{cite}
\usepackage{hyperref}
\usepackage{mathrsfs}
\usepackage{tikz}
\usetikzlibrary{patterns}
\usepackage{caption}
\DeclareCaptionLabelSeparator{dot}{. }
\newcommand{\al}{\alpha}
\newcommand{\be}{\beta}
\newcommand{\ga}{\gamma}
\newcommand{\de}{\delta}
\newcommand{\la}{\lambda}
\newcommand{\om}{\omega}

\newcommand{\eps}{\varepsilon}
\newcommand{\vv}{\varphi}

\theoremstyle{plain}

\numberwithin{equation}{section}

\newtheorem{thm}{Theorem}[section]
\newtheorem{lem}[thm]{Lemma}
\newtheorem{prop}[thm]{Proposition}
\newtheorem{cor}[thm]{Corollary}

\theoremstyle{definition}

\newtheorem{ip}[thm]{Inverse Problem}

\theoremstyle{remark}

\DeclareMathOperator*{\Res}{Res}

 \allowdisplaybreaks

\begin{document}

\begin{center}
{\Large\bf Inverse spectral problem \\[0.2cm] for the third-order differential equation}
\\[0.5cm]
{\bf Natalia P. Bondarenko}
\end{center}

\vspace{0.5cm}

{\bf Abstract.} This paper is concerned with the inverse spectral problem for the third-order differential equation with distribution coefficient. The inverse problem consists in the recovery of the differential expression coefficients from the spectral data of two boundary value problems with separated boundary conditions. For this inverse problem, we solve the most fundamental question of the inverse spectral theory about the necessary and sufficient conditions of solvability. In addition, we prove the local solvability and stability of the inverse problem. Furthermore, we obtain very simple sufficient conditions of solvability in the self-adjoint case. The main results are proved by a constructive method that reduces the nonlinear inverse problem to a linear equation in the Banach space of bounded infinite sequences. In the future, our results can be generalized to various classes of higher-order differential operators with integrable or distribution coefficients.

\medskip


\medskip


\vspace{1cm}

\section{Introduction} \label{sec:intr}

This paper deals with the third-order differential equation
\begin{equation} \label{eqv}
y''' + (\tau_1(x) y)' + \tau_1(x) y' + \tau_0(x) y = \la y, \quad x \in (0,1),
\end{equation}
where $\la$ is the spectral parameter, $\tau_1 \in L_2(0,1)$, $\tau_0 \in W_2^{-1}(0,1)$, that is, $\tau_0 = \sigma_0'$, $\sigma_0 \in L_2(0,1)$, and the derivative of $L_2$-function is understood in the sense of distributions.

We study the recovery of the coefficients $\tau_0$ and $\tau_1$ from the spectral data of the two boundary value problems $\mathcal L_1$ and $\mathcal L_2$ for  equation \eqref{eqv} with the following boundary conditions:
\begin{align} \label{bc1}
\mathcal L_1 \colon & \quad y(0) = 0, \quad y(1) = y'(1) = 0, \\ \label{bc2}
\mathcal L_2 \colon & \quad y(0) = y'(0) = 0, \quad y(1) = 0.
\end{align}

In recent years, spectral theory of the third-order differential operators with non-smooth and distributional coefficients attracts considerable attention of scholars (see, e.g., \cite{BP19, Kor19, Ug19, Ug20, BK21, Ao22-3ord, ZLW23}). The third-order differential equations arise in various physical applications, e.g., in modelling thin membrane flow of viscous liquid and elastic beam vibrations (see \cite{Greg87, BP96, TS90}). The third-order operators play an important role in the integration of the nonlinear Boussinesq equation (see \cite{McK81}).

The paper is concerned with the theory of inverse spectral problems, which consist in the recovery of differential operators from their spectral characteristics. The greatest success in inverse spectral theory has been achieved for the second-order Sturm-Liouville operator $-y'' + q(x) y$ (see the monographs \cite{Mar77, Lev84, FY01, Krav20} and references therein). The basic results for the inverse Sturm-Liouville problems were obtained by the Gelfand-Levitan method \cite{GL51}. However, this method appeared to be ineffective for the higher-order differential operators:
\begin{equation} \label{ho}
y^{(n)} + \sum_{k = 0}^{n-2} p_k(x) y^{(k)}, \quad n > 2.
\end{equation}

Therefore, the investigation of inverse problems for operators \eqref{ho} required the development of new approaches. The general inverse problem theory for the higher-order differential operators \eqref{ho} with integrable coefficients $p_k$ on a finite interval and on the half-line has been constructed by Yurko \cite{Yur92, Yur95-mn, Yur00, Yur02} by the method of spectral mappings. The central idea of this method consists in the reduction of a nonlinear inverse problem to a linear equation in a suitable Banach space. The main technical tool is the contour integration of specific functions called the spectral mappings, which first appeared in the papers of Leibenson \cite{Leib66, Leib71}. Inverse scattering problems for higher-order differential operators on the full line were considered by Beals \cite{Beals85} and his research group.

Recently, the study of inverse spectral problems began for higher-order differential operators with distribution coefficients. In \cite{Bond21, Bond22-half}, the uniqueness theorems for the recovery of such operators on a finite interval and on the half-line have been proved. In \cite{Bond22-alg}, a constructive approach for solving inverse spectral problems has been developed. This approach can be applied to various classes of differential operators with regular or distributional coefficients. The methods of \cite{Bond21, Bond22-half, Bond22-alg} rely on the  regularization of differential operators with distribution coefficients (see \cite{MS16, MS19, Vlad17}) and on the ideas of the method of spectral mappings \cite{Yur02}. In the reconstruction technique of \cite{Bond22-alg}, an important role is played by the spectral data asymptotics which have been deduced in \cite{Bond23-asympt} by using the Birkhoff-type solutions recently obtained in \cite{SS20}. However, necessary and sufficient conditions (NSC) of inverse problem solvability for higher-order differential operators with distribution coefficients, to the best of the author's knowledge, have not been investigated yet.

Note that the NSC question is the most fundamental issue in the theory of inverse spectral problems. At the same time, this issue is usually the most difficult for investigation. As an example, let us consider the following well-known result for the Sturm-Liouville problem:
\begin{gather} \label{StL}
    -y'' + q(x) y = \la y, \quad x \in (0,1), \\ \label{StLbc}
    y'(0) - h y(0) = 0, \quad y'(1) + H y(1) = 0,
\end{gather}
where $q(x)$ is a real-valued function of $L_2(0,1)$, $h$ and $H$ are real constants. Denote by $\{ \la_n \}_{n = 1}^{\infty}$ and $\{ y_n(x) \}_{n = 1}^{\infty}$ the eigenvalues of the problem \eqref{StL}-\eqref{StLbc} and the corresponding eigenfunctions normalized by the condition $y_n(0) = 1$. 
Consider the spectral data $\{ \la_n, \al_n \}_{n = 1}^{\infty}$ which consist of the eigenvalues and
the weight numbers $\al_n := \int_0^1 y_n^2(x) \, dx$, $n \in \mathbb N$.

\begin{prop}[\cite{GL51, FY01}] \label{prop:StL}
For numbers $\{ \la_n, \al_n \}_{n = 1}^{\infty}$ to the the spectral data of some Sturm-Liouville problem of form \eqref{StL}-\eqref{StLbc}, the following conditions are necessary and sufficient:
\begin{gather} \label{nsc1}
    \la_n \in \mathbb R, \quad \la_n \ne \la_m, \: n \ne m, \quad \al_n > 0, \\ \label{nsc2}
    \sqrt{\la_n} = \pi n + \frac{\om}{\pi n} + \frac{\varkappa_n}{n}, \quad
    \al_n = \frac{1}{2} + \frac{\varkappa_{n1}}{n},
\end{gather}
where $\om = \frac{1}{2}\int_0^1 q(x) \, dx$, $\{ \varkappa_n \} \, \{ \varkappa_{n1} \} \in l_2$. 
\end{prop}

Proposition~\ref{prop:StL} is remarkable by the concise form of the NSC, which include only simple structural properties \eqref{nsc1} and the asymptotics \eqref{nsc2}. For the first time, NSC of the inverse Sturm-Liouville problem solvability have been obtained in the seminal paper by Gelfand and Levitan \cite{GL51}. However, the results of \cite{GL51} were slightly different from Proposition~\ref{prop:StL}, because in \cite{GL51} there was a gap between the necessary conditions and the sufficient ones. By the sufficiency, more precise asymptotics were required. Later on, that gap has been removed. The NSC of Proposition~\ref{prop:StL} without the gap can be found, e.g., in \cite{FY01}. This historical example shows that, even for the simplest second-order operator, obtaining NSC on the spectral data required considerable effort.

For the higher-order differential operators \eqref{ho} with regular coefficients, the NSC of the inverse problem solvability have been obtained by Yurko (see \cite[Theorem~2.3.1]{Yur02}). However, in contrast to Proposition~\ref{prop:StL}, Theorem~2.3.1 of \cite{Yur02} contains several hard-to-verify conditions. First, it requires the existence of a model problem whose spectral data are asymptotically close in some sense to the given data. Second, the unique solvability of the main linear equation is required. Although this condition is unavoidable for non-self-adjoint operators, it is important to study the special cases when the main equation solvability can be deduced from some easy-to-verify conditions. Third, the reconstruction formulas for the differential expression coefficients in \cite{Yur02} have the form of series and, in Theorem~2.3.1, the a posteriori requirement of the series convergence in the appropriate spaces is imposed. The reason of this last requirement is that, in Theorem~2.3.1, the coefficients $p_k$ of \eqref{ho} belong to non-Hilbert spaces. In Theorem~2.3.3 of \cite{Yur02}, the NSC without such a posteriori requirement for the Hilbert space case are provided without proofs. Thus, despite the fact that Yurko's results were a great advance in the theory of inverse problems for higher-order differential operators, these results are not final and the study of such problems need to be continued. We also emphasize that Yurko's results concern only the case of regular coefficients. For differential operators with distribution coefficients, the NSC question is completely open. In the author's opinion, the study of the distribution coefficient case will help not only to construct the general inverse problem theory of higher-order differential operators but also to overcome some difficulties which arise in the case of regular coefficients.

The goal of this paper is to obtain the NSC on the spectral data of the third-order equation \eqref{eqv}. The third order is chosen because of the two reasons. First, for the third-order equation \eqref{eqv}, the inverse problem solvability conditions have been obtained in the simplest form (see Theorem~\ref{thm:suff}), without any hard-to-verify requirements. Second, the third order is convenient for presentation of the proof technique, which in the future can be generalized to arbitrary orders. We treat the differential equation \eqref{eqv} with distribution coefficient in terms of the Mirzoev-Shkalikov approach \cite{MS16, MS19}. For simplicity, we choose the boundary conditions \eqref{bc1}-\eqref{bc2} of the lowest possible orders. They are analogous to the Dirichlet boundary conditions $y(0) = y(1) = 0$ for the Sturm-Liouville operator. The other types of separated boundary conditions can be studied similarly. Non-separated boundary conditions (e.g., the periodic ones) are fundamentally different and so require a separate investigation.

As spectral data, we take the eigenvalues $\{ \la_{n,k} \}$ and the weight numbers $\{ \be_{n,k} \}$ of the two boundary value problems $\mathcal L_k$, $k = 1, 2$, and, for any pair of coinciding eigenvalues $\la_{n,1} = \la_{n,2}$, their ``common'' weight number $\ga_n$ is added. The rigorous definition of the spectral data is provided in Section~\ref{sec:main}. For solving the inverse problem, we use a constructive approach of \cite{Yur02, Bond22-alg}, which reduces the inverse problem to the so-called main equation in the Banach space of bounded infinite sequences. Applying this method, we prove the main theorem (Theorem~\ref{thm:nsc}) on the NSC of the inverse problem solvability in the general non-self-adjoint case. More precisely, using the main equation solution, we construct the functions $\tau_1$ and $\tau_0$ as some series. Then, we prove the convergence of these series in the spaces $L_2(0,1)$ and $W_2^{-1}(0,1)$, respectively, relying on the spectral data asymptotics. Thereafter, we show that the spectral data of the boundary value problems $\mathcal L_1$ and $\mathcal L_2$ with the constructed coefficients $\tau_0$ and $\tau_1$ coincide with the initially given numbers having prescribed asymptotical and structural properties.
Furthermore, we consider the case of a small perturbation of the spectral data and obtain the local solvability and stability of the inverse problem (Theorem~\ref{thm:loc}). Finally, we investigate the self-adjoint case, when the functions $\mathrm{i} \tau_0$ and $\tau_1$ are real-valued. For this case, we obtain simple sufficient conditions of the inverse problem solvability (Theorem~\ref{thm:suff}). The central role in the proofs of Theorem~\ref{thm:suff} belongs to the unique solvability of the main equation (Lemma~\ref{lem:solve}). In order to prove that lemma, we develop a new technique, which has no analogs in previous studies, because the main equation solvability for odd orders has not been investigated before.

The paper is organized as follows. In Section~\ref{sec:main}, the spectral data are defined, the main results are presented, and the proof strategy is briefly described. In Section~\ref{sec:sd}, we study structural and asymptotical properties of the spectral data. In Section~\ref{sec:equ}, we provide the construction of the inverse problem main equation from \cite{Bond22-alg}. Section~\ref{sec:proof} contains the proofs of Theorem~\ref{thm:nsc} on the NSC and of Theorem~\ref{thm:loc} on the local solvability and stability of the inverse problem. In Section~\ref{sec:sa}, the self-adjoint case is considered and Theorem~\ref{thm:suff} on the sufficient conditions of the inverse problem solvability for this case is proved. In Section~\ref{sec:concl}, we briefly summarize our results and discuss the possibility of generalizing them to arbitrary orders.

Throughout this paper, we use the following \textbf{notations}:
\begin{itemize}
    \item The prime $y'(x, \la)$ denotes the derivative with respect to $x$ and the dot $\dot y(x, \la)$, with respect to $\la$.
    \item $\de_{k,j}$ is the Kronecker delta.
    \item In estimates, the same symbol $C$ is used for various positive constants independent of $x$, $\la$, $n$, etc.
    \item If for $\la \to \la_0$
    $$
    A(\la) = \sum_{k = -q}^p a_k(\la - \la_0)^k + o((\la-\la_0)^p),
    $$
    then $A_{\langle k \rangle}(\la_0) := a_k$.
    \item Along with the differential expression $\ell(y) = y''' + (\tau_1(x) y)' + \tau_1(x) y' + \tau_0(x) y$, we consider the differential expressions $\tilde \ell$, $\ell^{\star}$, $\tilde \ell^{\star}$, $\ell^N$, and $\ell^{\dagger}$ of the analogous form but with different coefficients. We agree that, if a symbol $\al$ denotes an object related to $\ell$, then the symbols $\tilde \al$, $\al^{\star}$, $\tilde \al^{\star}$, $\al^N$, and $\al^{\dagger}$  will denote the analogous objects related to $\tilde \ell$, $\ell^{\star}$, $\tilde \ell^{\star}$, $\ell^N$, and $\ell^{\dagger}$, respectively. Note that the quasi-derivatives $y^{[j]}$ for $\ell$, $\tilde \ell$, $\ell^{\star}$, etc., will be defined differently.
\end{itemize}

\section{Main results} \label{sec:main}

We start with the regularization of equation \eqref{eqv}.
Following the approach of Mirzoev and Shkalikov \cite{MS16, MS19}, we understand equation \eqref{eqv} in terms of quasi-derivatives. Let us briefly describe this approach.

The associated matrix $F(x) = [f_{k,j}(x)]_{k,j = 1}^3$ of equation \eqref{eqv} has the form
\begin{equation} \label{defF}
F(x) = \begin{bmatrix}
            0 & 1 & 0 \\
            -(\sigma_0 + \tau_1) & 0 & 1 \\
            0 & (\sigma_0 - \tau_1) & 0
        \end{bmatrix},
\end{equation}
where $\sigma_0$ is any fixed antiderivative of $\tau_0$. Note that $\sigma_0$ can be chosen up to a constant. Nevertheless, the choice of $\sigma_0$ does not influence on the spectral data, which are defined below in this section.

The quasi-derivatives are defined by the formulas
\begin{equation} \label{quasi}
y^{[0]} := y, \quad y^{[k]} := (y^{[k-1]})' - \sum_{j = 1}^k f_{k,j} y^{[j - 1]}, \quad k = 1, 2, 3.
\end{equation}
Thus
$$
y^{[j]} = y^{(j)}, \quad j = 0,1, \quad y^{[2]} = y'' + (\sigma_0 + \tau_1) y, \quad y^{[3]} = (y^{[2]})' - (\sigma_0 - \tau_1) y'.
$$
Define the domain 
$$
\mathcal D_F := \{ y \colon y^{[k]} \in AC[0,1], \, k = 0,1,2 \}.
$$

It follows from the results of \cite{MS19} that, for any $y \in \mathcal D_F$, the differential expression $\ell(y) = y''' + (\tau_1(x) y)' + \tau_1(x) y + \tau_0(x) y$ produces a regular generalized function and $\ell(y) = y^{[3]}$. Therefore, a function $y$ is called a solution of equation \eqref{eqv} if $y \in \mathcal D_F$ and $y^{[3]} = \la y$ a.e. on $(0,1)$.

Denote by $C_k(x, \la)$, $k = 1, 2, 3$, the solutions of equation \eqref{eqv} satisfying the initial conditions
$$
C_k^{[j-1]}(0,\la) = \de_{k,j}, \quad k,j = 1, 2, 3,
$$

Obviously, the functions $C_k(x, \la)$ are uniquely defined as the solutions of the following initial value problems:
\begin{equation} \label{initC}
\begin{bmatrix}
    C_k \\ C_k' \\ C_k^{[2]}
\end{bmatrix}(0,\la) = 
\begin{bmatrix}
    \de_{k,1} \\ \de_{k,2} \\ \de_{k,3}
\end{bmatrix}, \quad
\frac{d}{dx}
\begin{bmatrix}
    C_k \\ C_k' \\ C_k^{[2]}
\end{bmatrix} = 
\begin{bmatrix}
    0 & 1 & 0 \\
    -(\sigma_0 + \tau_1) & 0 & 1 \\
    \la & (\sigma_0 - \tau_1) & 0
\end{bmatrix}
\begin{bmatrix}
    C_k \\ C_k' \\ C_k^{[2]}
\end{bmatrix}, \quad k = 1, 2, 3,
\end{equation}

Consequently, the quasi-derivatives $C_k^{[j-1]}(x, \la)$ are entire in $\la$ for each fixed $x \in [0,1]$, $k,j = 1,2,3$.
Define the entire functions
\begin{align} \label{defDelta1}
& \Delta_{1,1}(\la) := -\begin{vmatrix}
                        C_2(1,\la) & C_3(1,\la) \\
                        C_2'(1,\la) & C_3'(1,\la)
                    \end{vmatrix}, \quad
\Delta_{2,1}(\la) := -\begin{vmatrix}
                        C_1(1,\la) & C_3(1,\la) \\
                        C_1'(1,\la) & C_3'(1,\la)
                    \end{vmatrix}, \\ \label{defDelta2}
& \Delta_{3,1}(\la) := \begin{vmatrix}
                        C_1(1,\la) & C_2(1,\la) \\
                        C_1'(1,\la) & C_2'(1,\la)
                    \end{vmatrix}, \quad 
\Delta_{2,2}(\la) := C_3(1, \la), \quad \Delta_{3,2}(\la) := C_2(1,\la).
\end{align}

Proceed with the definition of the spectral data.
For $k =1, 2$, denote by $\{ \la_{n,k} \}_{n = 1}^{\infty}$ the eigenvalues of the corresponding boundary value problem $\mathcal L_k$. 
One can easily check that the eigenvalues $\{ \la_{n,k} \}_{n = 1}^{\infty}$ coincide with the zeros of the characteristic function $\Delta_{k,k}(\la)$, $k = 1,2$. Throughout the paper, we assume that the zeros of $\Delta_{k,k}(\la)$ are simple for $k = 1, 2$. The case of multiple eigenvalues can be studied by using the ideas of the papers \cite{But07, BSY13}, in which the inverse spectral problems have been investigated for the non-self-adjoint Sturm-Liouville operators. However, for the higher-order differential operators, the case of multiple eigenvalues is much more technically complicated, so we confine ourselves to the case of simple eigenvalues.
Anyway, it is possible that $\la_{n,1} = \la_{p,2}$ for some indices $n,p \ge 1$. In this case, we reorder the eigenvalues so that $n = p$ and define the set 
\begin{equation} \label{defK}
K := \{ n \in \mathbb N \colon \la_{n,1} = \la_{n,2} \}.
\end{equation}
For $n \not \in K$, we have $\la_{n,1} \not \in \{ \la_{p,2} \}_{p =  1}^{\infty}$.

Together with the eigenvalues, we will use additional spectral information. In the inverse problem theory of the higher-order differential operators, the most natural spectral characteristics is the Weyl-Yurko matrix, which generalizes the Weyl functions of the Sturm-Liouville operators (see \cite{Mar77, FY01}). For the higher-order differential operators with regular coefficients, the Weyl-Yurko matrix for the first time was introduced by Yurko \cite{Yur92, Yur02}. The Weyl-Yurko matrix uniquely specifies the differential operator in the general case, while the spectral data used, e.g., in \cite{Leib66, Leib71, Beals85} are sufficient only under some restrictions on the spectra. For differential operators with distribution coefficients, the Weyl-Yurko matrices were used in \cite{Bond21, Bond22-half, Bond22-alg}.

For the third-order equation \eqref{eqv}, the Weyl-Yurko matrix is defined as follows:
\begin{equation} \label{defM}
M(\la) = \begin{bmatrix}
            1 & 0 & 0 \\
            M_{2,1} & 1 & 0 \\
            M_{3,1} & M_{3,2} & 1
         \end{bmatrix}, \quad
 M_{j,k}(\la) = -\frac{\Delta_{j,k}(\la)}{\Delta_{k,k}(\la)}, \quad 1 \le k < j \le 3.
\end{equation}

Clearly, the so-called Weyl functions $M_{j,k}(\la)$, $1 \le k < j \le 3$, are meromorphic in $\la$ and their poles coincide with the eigenvalues $\{ \la_{n,k} \}_{n = 1}^{\infty}$. 
Define the weight numbers
\begin{equation} \label{defbeta}
\beta_{n,k} := -\Res_{\la =\la_{n,k}} M_{k+1,k}(\la) = \frac{\Delta_{k+1,k}(\la_{n,k})}{\dot \Delta_{k,k}(\la_{n,k})}, \quad n \in \mathbb N, \, k = 1,2,
\end{equation}
and $\dot\Delta(\la) = \frac{d}{d\la} \Delta(\la)$. It can be easily shown (see Lemma~\ref{lem:weight}) that $\beta_{n,1} \beta_{n2} = 0$ if and only if $n \in K$. Therefore, for $n \in K$, we put $\la_n := \la_{n,1}$ and define additional weight numbers
\begin{align} \label{defga1}
& \ga_n := \dfrac{\Delta_{3,1}(\la_n)}{\dot \Delta_{1,1}(\la_n)} \quad \text{if} \:\: \beta_{n,1} = 0,  \\ \label{defga2}
& \ga_n := \dfrac{C_1(1,\la_n)}{\dot \Delta_{2,2}(\la_n)} \quad \text{if} \:\: \beta_{n,2} = 0.
\end{align}

If $\be_{n,1} = 0$ and $\be_{n,2} = 0$, then the definitions \eqref{defga1} and \eqref{defga2} coincide with each other. We use the eigenvalues of the two problems $\mathcal L_1$, $\mathcal L_2$ together with the defined weight numbers as the spectral data of the inverse problem.

\begin{ip} \label{ip:main}
Given the spectral data 
$$
\mathfrak S := \bigl( \{ \la_{n,k} \}_{n \in \mathbb N, \, k = 1, 2}, \{ \beta_{n,k} \}_{n \in \mathbb N, \, k = 1, 2}, \{ \ga_n \}_{n \in K}\bigr), 
$$
find the coefficients $\mathcal T := (\tau_0, \tau_1)$ of equation \eqref{eqv}.
\end{ip}

We will write that $\mathcal T \in W$ if $\tau_0 \in W_2^{-1}(0,1)$, $\tau_1 \in L_2(0,1)$ and the zeros of $\Delta_{k,k}(\la)$ are simple for $k = 1, 2$. Along with $\mathcal T$, we consider another coefficient pair $\tilde{\mathcal T} = (\tilde \tau_0, \tilde \tau_1)$ of class $W$. We agree that, if a symbol $\al$ denotes an object related to $\mathcal T$, then the symbol $\tilde \al$ with tilde  will denote the analogous object related to $\tilde {\mathcal T}$. Thus, the uniqueness theorem for Inverse Problem~\ref{ip:main} is formulated as follows.

\begin{thm} \label{thm:uniq}
If $\mathfrak S = \tilde{\mathfrak S}$, then $\mathcal T = \tilde {\mathcal T}$, that is, $\tau_0 = \tilde \tau_0$ in $W_2^{-1}(0,1)$, $\tau_1 = \tilde \tau_1$ in $L_2(0,1)$. Thus, the coefficients $\mathcal T$ of class $W$ are uniquely specified by the spectral data $\mathfrak S$.
\end{thm}

Inverse Problem~\ref{ip:main} can be solved constructively by the method of \cite{Bond22-alg}. Namely, the inverse problem can be reduced to the linear main equation of form
\begin{equation} \label{main-intr}
(I - \tilde R(x)) \psi(x) = \tilde \psi(x), \quad x \in [0,1],
\end{equation}
in the space $m$ of infinite bounded sequences. The construction of the main equation is provided in detail in Section~\ref{sec:equ}. Here, we only mention that, for each fixed $x \in [0,1]$, $\psi(x)$ and $\tilde \psi(x)$ are elements of $m$, $\tilde R(x) \colon m \to m$ is a compact linear operator, and $I$ is the unit operator in $m$. The element $\tilde \psi(x)$ and the operator $\tilde R(x)$ are constructed by the spectral data $\mathfrak S$, while $\psi(x)$ is the unknown element, which is related to the coefficients $\mathcal T = (\tau_0, \tau_1)$ of equation \eqref{eqv}. Thus, solving the main equation \eqref{main-intr}, one can find $\psi(x)$ and then use it to obtain the solution of the inverse problem.

Now we are ready to formulate the main theorem on the necessary and sufficient conditions of the inverse problem solvability.

\begin{thm} \label{thm:nsc}
For numbers $\mathfrak S = \bigl( \{ \la_{n,k} \}_{n \in \mathbb N, \, k = 1, 2}, \{ \beta_{n,k} \}_{n \in \mathbb N, \, k = 1, 2}, \{ \ga_n \}_{n \in K}\bigr)$ to be the spectral data corresponding to $\mathcal T = (\tau_0, \tau_1) \in W$, it is necessary and sufficient to fulfill the following conditions:
\begin{enumerate}
    \item The following asymptotic formulas hold for $n \ge 1$, $k = 1, 2$:
    \begin{equation} \label{asymptla}
        \la_{n,k} = (-1)^{k+1} \left( \frac{2\pi}{\sqrt 3} \Bigl( n + \frac{1}{6} - \frac{\theta}{2\pi^2 n} + \frac{\varkappa_n}{n}  \Bigr)\right)^3, \quad \beta_{n,k} = 3 \la_{n,k} \left( 1 + \frac{\varkappa_{n1}}{n}\right), 
    \end{equation}
    where $\theta = \int\limits_0^1 \tau_1(x) \, dx$, $\{ \varkappa_n \}, \{ \varkappa_{n1} \} \in l_2$.
    \item $\la_{n,k} \ne \la_{n_0,k_0}$ if $n \ne n_0$;
    $\beta_{n,1} \beta_{n,2} = 0$ if and only if $n \in K$; $\ga_n \ne 0$ for $n \in K$.
    \item For each fixed $x \in [0,1]$, the operator $(I - \tilde R(x)) \colon m \to m$ of the main equation \eqref{main-intr} has a bounded inverse.
\end{enumerate}
\end{thm}

The first condition of Theorem~\ref{thm:nsc} means that the values $\{ \la_{n,k} \}_{n \in\mathbb N, \, k = 1, 2}$ can be numbered so that the asymptotic formulas \eqref{asymptla} are valid. Such numbering is non-unique, since the asymptotics determines only the order of sufficiently large eigenvalues. We fix a numbering satisfying this property and also the condition $n = p$ if $\la_{n,1} = \la_{p,2}$. This numbering specifies the set $K$ by \eqref{defK}. By virtue of the asymptotics \eqref{asymptla}, the values $\{ \la_{n,k} \}$ are separated for sufficiently large values of $n$, so the set $K$ is finite and the assumption of simplicity automatically holds for large eigenvalues. Therefore, this assumption is not very restrictive.

Let us briefly describe the proof strategy of Theorem~\ref{thm:nsc}.
The proof of the necessity consists in the study of structural and asymptotical properties of the spectral data in Section~\ref{sec:sd}. The necessity of the third condition has been proved in \cite{Bond22-alg}. The proof of the sufficiency is based on the constructive solution of Inverse Problem~\ref{ip:main}. Since the operator $(I - \tilde R(x))$ is invertible, then the main equation \eqref{main-intr} has the unique solution $\psi(x) \in m$. Using the entries of $\psi(x)$, we construct the functions $\tau_0$ and $\tau_1$ by the reconstruction formulas derived in \cite{Bond22-alg}. We show that $\tau_0 \in W_2^{-1}(0,1)$ and $\tau_1 \in L_2(0,1)$. Finally, we prove that the initially given numbers $\mathfrak S$ are the spectral data of equation \eqref{eqv} with the constructed coefficients $\tau_0$ and $\tau_1$. The last step is the most technically difficult. We need to check that some constructed series are solutions of equation~\eqref{eqv} and, moreover, fulfill certain boundary conditions. But the derivatives of these series do not converge and cannot be directly substituted into equation \eqref{eqv}. In order to overcome these difficulties, we use the approximation approach, which is described in detail in Section~\ref{sec:proof}.

Note that Theorem~\ref{thm:nsc} contains the condition of the main equation unique solvability (condition 3). Conditions of such kind usually required for non-self-adjoint operators. For example, the similar condition appears in the study of the non-self-adjoint Sturm-Liouville operators (see \cite[Theorem~1.6.3]{FY01}). Nevertheless, it is important to investigate special cases, when condition~3 can be deduced from some easy-to-verify conditions. In this paper, we consider the two such cases:
\begin{enumerate}
    \item Small perturbation of the spectral data.
    \item The self-adjoint case.
\end{enumerate}

For the case of a small perturbation, we obtain the following theorem on the local solvability and stability of Inverse Problem~\ref{ip:main}.

\begin{thm} \label{thm:loc}
Suppose that $\tilde {\mathcal T} = (\tilde \tau_0, \tilde \tau_1) \in W$ and $\tilde K = \varnothing$. Then, there exists $\eps > 0$ (which depends on $\tilde {\mathcal T}$) such that, for any complex numbers $\mathfrak S := \{ \la_{n,k}, \be_{n,k} \}_{n \in \mathbb N, \, k = 1, 2}$ satisfying the inequality
\begin{equation} \label{locsd}
d(\mathfrak S, \tilde{\mathfrak S}) := \sqrt{\sum_{n = 1}^{\infty} \sum_{k = 1, 2} \bigl( n^{-1} |\la_{n,k} - \tilde \la_{n,k}| + n^{-2} |\be_{n,k} - \tilde \be_{n,k}| \bigr)^2} \le \eps,
\end{equation}
there exist coefficients $\mathcal T = (\tau_0, \tau_1) \in W$ such that $\mathfrak S$ are the spectral data of $\mathcal T$. Moreover,
\begin{equation} \label{loctau}
 \| \tau_1 - \tilde \tau_1 \|_{L_2(0,1)} \le C d(\mathfrak S, \tilde{\mathfrak S}), \quad
 \| \tau_0 - \tilde \tau_0 \|_{W_2^{-1}(0,1)} \le C d(\mathfrak S, \tilde{\mathfrak S}),
\end{equation}
where the constant $C > 0$ depends only on $\tilde {\mathcal T}$ and $\eps$.
\end{thm}

The case $\tilde K \ne \varnothing$ requires a separate investigation, because the equal eigenvalues $\tilde \la_{n,1} = \tilde \la_{n,2}$ can split under a small perturbation of the spectra. In order to prove Theorem~\ref{thm:loc}, we show that \eqref{locsd} implies the conditions~1--3 of Theorem~\ref{thm:nsc}.

Now proceed to the self-adjoint case.
Suppose that the functions $\mathrm{i} \tau_0(x)$ and $\tau_1(x)$ are real-valued. Then, the eigenvalues of the boundary value problems $\mathcal L_k$, $k = 1, 2$, can be numbered so that
\begin{equation} \label{sala}
\la_{n,1} = -\overline{\la_{n,2}}, \quad \beta_{n,1} = -\overline{\beta_{n,2}}, \quad n \in \mathbb N.
\end{equation}

Strictly speaking, the differential expression $\ell(y) = y''' + (\tau_1(x) y)' + \tau_1(x) y' + \tau_0(x) y$ becomes self-adjoint when multiplied by $\mathrm{i}$. Then the boundary value problems $\mathcal L_1$ and $\mathcal L_2$ become adjoint to each other, so their spectra are complex conjugate to each other. However, in order to preserve similarity of notations with the previous studies \cite{Yur02, Bond22-alg}, we do not multiply by $\mathrm{i}$. Thus, the spectra of $\mathcal L_1$ and $\mathcal L_2$ are symmetric to each other with respect to the imaginary axis.

In view of \eqref{sala}, it is sufficient to consider only the spectrum $\{ \la_n \}_{n = 1}^{\infty}$, $\la_n := \la_{n,1}$, $\beta_n := \beta_{n,1}$. As before, assume that the zeros $\{ \la_n \}_{n = 1}^{\infty}$ of $\Delta_{1,1}(\la)$ are simple. Since the eigenvalues are numbered according to \eqref{sala}, it can happen that $\la_{n,1} = \la_{p,2}$ for $n \ne p$, that is, $\la_n = -\overline{\la_p}$. However, we are going to provide sufficient conditions of the inverse problem solvability, and we additionally assume that $\la_n \ne -\overline{\la_p}$ for $n \ne p$. Anyway, it is possible that $\la_n = -\overline{\la_n}$ for some $n \in \mathbb N$, then $\be_n = 0$ and $\ga_n > 0$, where $\ga_n$ can be defined by either \eqref{defga1} or \eqref{defga2}. Obviously $K = \{ n \in \mathbb N \colon \la_n = -\overline{\la_n} \}$. Thus, define the spectral data
\begin{equation} \label{defS+}
\mathfrak S^+ := \bigl( \{ \la_n \}_{n = 1}^{\infty}, \{ \beta_n \}_{n \in \mathbb N \setminus K}, \{ \ga_n \}_{n \in K} \bigr).
\end{equation}

We will say that $\mathcal T \in W^+$ if $\tau_1 \in L_2(0,1)$, $\tau_0 \in W_2^{-1}(0,1)$, the functions $\mathrm{i} \tau_0(x)$ and $\tau_1(x)$ are real-valued. 

\begin{thm} \label{thm:suff}
Let $\mathfrak S^+$ be arbitrary numbers of form \eqref{defS+} satisfying the conditions:
\begin{enumerate}
    \item The numbers $\{ \la_n \}_{n \in \mathbb N}$ and $\{ \beta_n \}_{n \in \mathbb N \setminus K}$ satisfy the asymptotics \eqref{asymptla} for $k = 1$.
    \item $\la_n \ne \la_p$ and $\la_n \ne -\overline{\la_p}$ if $n \ne p$; $\beta_n \ne 0$ for $n \in \mathbb N \setminus K$.
    \item $\mbox{Re} \, \la_n \ge 0$ for $n \in \mathbb N$ and $\ga_n > 0$ for $n \in K$.
\end{enumerate}
Then there exist a unique coefficient pair $\mathcal T = (\tau_0, \tau_1) \in W^+$ having the spectral data $\mathfrak S^+$.
\end{thm}

It is remarkable that Theorem~\ref{thm:suff} does not require the solvability of the main equation. Its unique solvability is proved in Section~\ref{sec:sa} by using the conditions 1--3 of Theorem~\ref{thm:suff}. The proof is based on the construction of special meromorphic functions, on the contour integration in the $\la$-plane, and on the application of the Residue Theorem. We emphasize that this construction is novel and different from the case of even-order differential operators. For even orders, the number of the boundary value problems $\mathcal L_k$ is odd, and the middle problem is self-adjoint. Consequently, the ideas analogous to the study of the second-order operators can be applied. For odd orders, there is no such ``middle'' self-adjoint problem. Therefore, we impose the condition $\mbox{Re} \, \la_n \ge 0$ for the two spectra to be separated by the imaginary axis. This separation plays an important role in the proof.

\section{Spectral data properties} \label{sec:sd}

In this section, we study structural and asymptotical properties of the spectral data $\mathfrak S$. In addition, we deduce the uniqueness theorem (Theorem~\ref{thm:uniq}) from the results of \cite{Bond22-alg}. Furthermore, we define the Weyl solutions $\Phi_k(x, \la)$ and $\Phi_k^{\star}(x, \la)$, $k = 1, 2, 3$, which play an important role in the investigation of the inverse problem. The relationship between the Weyl solutions and the spectral data is established.

Consider the boundary value problems $\mathcal L_k$, $k =1, 2$, defined in Section~\ref{sec:intr} for equation \eqref{eqv} with the coefficients $\mathcal T = (\tau_0, \tau_1) \in W$.

\begin{lem} \label{lem:weight}
$\beta_{n,1} \beta_{n,2}= 0$ if and only if $n \in K$. For $n \in K$, $\ga_n \ne 0$. If $\beta_{n,1} = \beta_{n,2}= 0$, then the definitions \eqref{defga1} and \eqref{defga2} coincide with each other.
\end{lem}

\begin{proof}
Suppose that $\la_{n,1} = \la_{n,2} =: \la_n$. Then $\Delta_{1,1}(\la_n) = \Delta_{2,2}(\la_n) = 0$. Hence, the relations \eqref{defDelta1}-\eqref{defDelta2} imply $C_3(1,\la_n) = 0$ and $C_2(1,\la_n) C_3'(1,\la_n) = 0$. Consider the two possible cases:

\smallskip

\textsc{Case 1.} $C_2(1,\la_n) = 0$, that is, $\Delta_{3,2}(\la_n) = 0$. Therefore, the Weyl function $M_{3,2}(\la) = -\dfrac{\Delta_{3,2}(\la)}{\Delta_{2,2}(\la)}$ is analytic at $\la = \la_n$, so $\beta_{n,2} = 0$. Note that, in this case, $C_1(1, \la_n) \ne 0$. Otherwise $C_1(1,\la_n) = C_2(1,\la_n) = C_3(1,\la_n) = 0$, which is impossible, since the functions $C_k(x, \la_n)$, $k =1, 2, 3$, form a fundamental solution system of equation \eqref{eqv}. Hence $\ga_n = \dfrac{C_1(1,\la_n)}{\dot \Delta_{2,2}(\la_n)} \ne 0$.

\smallskip

\textsc{Case 2.} $C_3'(1,\la_n) = 0$ implies that $\Delta_{2,1}(\la_n) = 0$, so the Weyl function $M_{2,1}(\la) = -\dfrac{\Delta_{2,1}(\la)}{\Delta_{1,1}(\la)}$ is analytic at $\la = \la_n$. Hence $\beta_{n,1} = 0$. Then
$$
\begin{vmatrix}
    C_1 & C_2 & C_3 \\
    C_1' & C_2' & C_3' \\
    C_1^{[2]} & C_2^{[2]} & C_3^{[2]} 
\end{vmatrix}(1,\la_n) = -C_1^{[2]}(1,\la_n) \Delta_{1,1}(\la_n) + C_2^{[2]}(1,\la_n) \Delta_{2,2}(\la_n) + C_3^{[2]}(1,\la_n) \Delta_{3,1}(\la_n)= 1. 
$$
Since $\Delta_{1,1}(\la_n) = \Delta_{2,1}(\la_n) = 0$, then $\Delta_{3,1}(\la_n) \ne 0$, so $\ga_n = \dfrac{\Delta_{3,1}(\la_n)}{\dot \Delta_{1,1}(\la_n)} \ne 0$. If additionally $\beta_{n,2}= 0$, then $C_2(1,\la_n) = 0$. Therefore,
$$
\Delta_{3,1}(\la_n) = C_1(1,\la_n) C_2'(1,\la_n), \quad 
\dot \Delta_{1,1}(\la_n) = \dot C_3(1,\la_n) C_2'(1,\la_n).
$$
Consequently, in this case, $\dfrac{\Delta_{3,1}(\la_n)}{\dot \Delta_{1,1}(\la_n)} = \dfrac{C_1(1,\la_n)}{\dot \Delta_{2,2}(\la_n)}$, that is, the definitions \eqref{defga1} and \eqref{defga2} of $\ga_n$ are equivalent to each other.

\smallskip

It can be shown similarly that $\beta_{n,1} \ne 0$ and $\beta_{n,2} \ne 0$ if $n \not\in K$.
\end{proof}

Recall that, under our assumptions, the eigenvalues $\{ \la_{n,k} \}$ are simple poles of the Weyl-Yurko meromorphic matrix function.
Consider the Laurent series
$$
M(\la) = \frac{M_{\langle -1 \rangle}(\la_{n,k})}{\la - \la_{n,k}} + M_{\langle 0 \rangle}(\la_{n,k}) + M_{\langle 1 \rangle}(\la_{n,k})(\la - \la_{n,k}) + \dots,
$$
and define the weight matrices
\begin{equation} \label{defN}
\mathcal N(\la_{n,k}) := (M_{\langle 0 \rangle}(\la_{n,k}))^{-1} M_{\langle -1\rangle}(\la_{n,k}).
\end{equation}

\begin{prop}[\cite{Bond22-alg}]\label{prop:uniq}  
Suppose that $\mathcal T = (\tau_0, \tau_1)$ and $\tilde {\mathcal T} = (\tilde \tau_0, \tilde \tau_1)$ belong to $W$ and $\la_{n,k} = \tilde \la_{n,k}$, $\mathcal N(\la_{n,k}) = \tilde{\mathcal N}(\la_{n,k})$ for all $n \in \mathbb N$, $k = 1, 2$. Then $\tau_0 = \tilde \tau_0$ in $W_2^{-1}(0,1)$ and $\tau_1 = \tilde \tau_1$ in $L_2(0,1)$. Thus, the spectral data $\{ \la_{n,k}, \mathcal N(\la_{n,k}) \}_{n \in \mathbb N, \, k = 1, 2}$ uniquely specify the coefficients $\tau_0$ and $\tau_1$ of equation \eqref{eqv}. 
\end{prop}

The following lemma establishes the relationship between the weight matrices $\mathcal N(\la_{n,k})$ and the weight numbers $\beta_{n,k}$, $\ga_n$.

\begin{lem} \label{lem:wm}
The following relations hold:
\begin{align*}
& n \not\in K \colon \quad \mathcal N(\la_{n,1}) = -\begin{bmatrix}
                            0 & 0 & 0 \\
                            \beta_{n,1} & 0 & 0 \\
                            0 & 0 & 0
                        \end{bmatrix},
\quad \mathcal N(\la_{n,2}) = -\begin{bmatrix}
                            0 & 0 & 0 \\
                            0 & 0 & 0 \\
                            0 & \beta_{n,2} & 0
                        \end{bmatrix}, \\
& n \in K \colon \quad \mathcal N(\la_n) = -\begin{bmatrix}
                            0 & 0 & 0 \\
                            \beta_{n,1} & 0 & 0 \\
                            \ga_n & \beta_{n,2} & 0
                        \end{bmatrix}.
\end{align*}
\end{lem}

\begin{proof}
The relation \eqref{defN} in the element-wise form yields
$$
\mathcal N(\la_{n,k}) = 
\begin{bmatrix}
    0 & 0 & 0 \\
    M_{2,1, \langle -1 \rangle} & 0 & 0 \\
    M_{3,1, \langle -1 \rangle} - M_{3,2,\langle 0 \rangle} M_{2,1, \langle -1 \rangle} & M_{3,2, \langle -1 \rangle} & 0
\end{bmatrix}(\la_{n,k}).
$$

By virtue of the definition \eqref{defbeta}, $M_{k+1,k \langle -1 \rangle}(\la_{n,k}) = -\beta_{n,k}$, $k =1, 2$. The technique similar to the proof of Lemma~\ref{lem:weight} shows that
$$
M_{3,1, \langle -1 \rangle}(\la_{n,k}) - M_{3,2,\langle 0 \rangle}(\la_{n,k}) M_{2,1, \langle -1 \rangle}(\la_{n,k}) = -
\begin{cases}
\ga_n, \:\: n \in K, \\
0,  \quad n \not\in K.  
\end{cases}
$$

This concludes the proof.
\end{proof}

Proposition~\ref{prop:uniq} together with Lemma~\ref{lem:wm} imply Theorem~\ref{thm:uniq} on the uniqueness of recovering $\mathcal T$ from the spectral data $\mathfrak S = (\{ \la_{n,k} \}, \{ \beta_{n,k} \}, \{ \ga_n \})$.

Proceed to the asymptotic properties of the spectral data. 

\begin{lem} \label{lem:asympt}
    The eigenvalues $\{ \la_{n,k} \}_{n \in \mathbb N, \, k = 1, 2}$ and the weight numbers $\{ \be_{n,k} \}_{n \in \mathbb N, \, k = 1, 2}$ satisfy the asymptotic relations \eqref{asymptla}.
\end{lem}

\begin{proof}
Along with \eqref{eqv}, consider the equation
$$
y''' + 2 \tau_1^0 y' = \la y, \quad x \in (0,1),
$$
where $\tau_1^0 = \int\limits_0^1 \tau_1(x) \, dx$, and the corresponding spectral data $\{ \la_{n,k}^0, \beta_{n,k}^0 \}_{n \in \mathbb N, \, k= 1, 2}$ generated by the boundary conditions \eqref{bc1} and \eqref{bc2}. The asymptotics of $\{ \la_{n,k}^0, \beta_{n,k}^0 \}_{n \in \mathbb N, \, k= 1, 2}$ can be found by the standard method (see, e.g., \cite[Chapter~II]{Nai68}):
\begin{gather*}
\la_{n,k}^0 = (-1)^{k+1} (\rho_{n,k}^0)^3, \quad
\rho_{n,k}^0 = \frac{2\pi}{\sqrt 3} \left( n + \frac{1}{6} - \frac{\tau_1^0}{2 \pi^2 n} + O\bigl( n^{-2}\bigr)\right), \\
\beta_{n,k}^0 = 3 \la_{n,k}^0 \left( 1 + O\bigl( n^{-2}\bigr)\right), \quad n \in \mathbb N, \quad k = 1, 2.
\end{gather*}

Applying Theorems~1.2 and 6.4 of \cite{Bond23-asympt} and taking the relation $\int_0^1 (\tau_1 - \tau_1^0)(x) \, dx = 0$ into account, we conclude that
\begin{gather*}
\la_{n,k} = (-1)^{k+1} \rho_{n,k}^3, \quad
\rho_{n,k} - \rho_{n,k}^0 = \frac{\varkappa_n}{n}, \quad \be_{n,k} - \be_{n,k}^0 = n^2 \varkappa_{n1}, \\
\{ \varkappa_n \}, \, \{ \varkappa_{n1} \} \in l_2, 
\quad n \in \mathbb N, \, k = 1, 2.
\end{gather*}

This yields the claim of the lemma.
\end{proof}

Denote by $\Phi_k(x, \la)$, $k = 1, 2, 3$, the so-called Weyl solutions of equation \eqref{eqv} satisfying the boundary conditions
$$
\Phi_k^{[j-1]}(0,\la) = \de_{k,j}, \quad j = \overline{1,k}, \quad
\Phi_k^{[3-j]}(1,\la) = 0, \quad j = \overline{k+1,3}.
$$

It can be easily shown that 
\begin{equation} \label{PhiCM}
\Phi_k(x, \la) = C_k(x, \la) + \sum_{j = k+1}^3 M_{j,k}(\la) C_j(x, \la), \quad k = 1, 2, 3.
\end{equation}

Consequently, the solutions $\Phi_1(x, \la)$ and $\Phi_2(x, \la)$ are meromorphic in $\la$ with the simple poles $\{ \la_{n,1} \}_{n = 1}^{\infty}$ and $\{ \la_{n,2} \}_{n = 1}^{\infty}$, respectively. The solution $\Phi_3(x, \la) \equiv C_3(x, \la)$ is entire in $\la$.
Using \eqref{PhiCM}, \eqref{defM}, \eqref{defbeta}, and \eqref{defga1}, one can easily establish the following connection between the Weyl solutions and the spectral data.

\begin{lem} \label{lem:WSD}
The following relations hold:
\begin{align*} 
& \Phi_2'(1, \la_{n,1}) = 0, \quad \Phi_3(1, \la_{n,2}) = 0, \quad n \not\in K, \\
& \be_{n,1} = -\Res\limits_{\la = \la_{n,1}} \Phi_1'(0, \la_{n,1}), \quad \be_{n,2} = -\Res\limits_{\la = \la_{n,2}} \Phi_2^{[2]}(0,\la), \quad n \in \mathbb N, \\ 
& \Phi_3(1,\la_n) = \Phi_3'(1,\la_n) = 0, \quad \ga_n = -\Res_{\la = \la_n} \Phi_1^{[2]}(0, \la), \quad n \in K, \, \text{if} \:\: \be_{n,1} = 0.
\end{align*}
\end{lem}

The Weyl solutions $\Phi_k(x, \la)$ are remarkable by their behavior as $|\la| \to \infty$. In order to describe this behavior, put $\la = \rho^3$ and divide the $\rho$-plane into the sectors $\Gamma_s = \left\{ \rho \in \mathbb C \colon \arg \rho \in \bigl( \frac{\pi (s-1)}{3}, \frac{\pi s}{3}\bigr)\right\}$, $s = \overline{1,6}$. In each fixed sector $\Gamma_s$, denote by $\{ \om_k \}_{k = 1}^3$ the roots of the equation $\om^3 = 1$ numbered so that
\begin{equation} \label{order}
\mbox{Re} \, (\rho \om_1) < \mbox{Re} \, (\rho \om_2) < \mbox{Re} \, (\rho \om_3), \quad \rho \in \Gamma_s.
\end{equation}

Clearly, the inequalities \eqref{order} become non-strict for $\rho \in \overline{\Gamma_s}$.

\begin{lem}\label{lem:estPhi}
In each closed sector $\overline{\Gamma_s}$, the following estimate is fulfilled
$$
|\Phi_k^{[j]}(x, \rho^3)| \le C |\rho|^{j - k + 1} |\exp(\rho \om_k x)|, \quad k = 1,2,3, \:\: j = 0,1,2, \:\: x \in [0,1], \:\: |\rho| \ge \rho^*.
$$
\end{lem}

The proof of Lemma~\ref{lem:estPhi} is based on expansions of the Weyl solutions by the fundamental system of the Birkhoff-type solutions. Such expansions have been estimated for differential operators of arbitrary order in \cite[Lemma 3]{Bond21} and in \cite[Proposition~2]{Bond22-alg}. It is worth mentioning that the estimates of Lemma~\ref{lem:estPhi} are similar to the ones for the case of differential operators with regular coefficients (see formulas (2.1.20) in \cite{Yur02}), and there are no principal differences in the proofs. Therefore, we omit the proof of Lemma~\ref{lem:estPhi}.

Along with $F(x)$, consider the matrix function
\begin{equation} \label{defF*}
     F^{\star}(x) := \begin{bmatrix}
            0 & 1 & 0 \\
            (\sigma_0 - \tau_1) & 0 & 1 \\
            0 & -(\sigma_0 + \tau_1) & 0
        \end{bmatrix},
\end{equation}
which defines the quasi-derivatives
\begin{equation} \label{quasiz}
z^{[0]} := z, \quad z^{[k]} := (z^{[k-1]})' - \sum_{j = 1}^k f^{\star}_{k,j} z^{[j - 1]}, \quad k = 1, 2, 3,
\end{equation}
the domain 
$$
\mathcal D_{F^{\star}} := \{ z \colon z^{[k]} \in AC[0,1], \, k = 0, 1, 2 \},
$$
and the differential expression $\ell^{\star}(z) = -z^{[3]}$. Analogously to $C_k(x, \la)$ and $\Phi_k(x, \la)$, we define the solutions $C_k^{\star}(x, \la)$ and $\Phi_k^{\star}(x, \la)$, $k = 1, 2, 3$, of the equation $\ell^{\star}(z) = \la z$ satisfying the initial conditions
$$
C_k^{\star[j-1]}(0, \la) = \de_{k,j}, \quad k,j = 1, 2, 3,
$$
and the boundary conditions
$$
\Phi_k^{\star [j-1]}(0,\la) = \de_{k,j}, \quad j = \overline{1,k}, \quad
\Phi_k^{\star [3-j]}(1,\la) = 0, \quad j = \overline{k+1,3},
$$
respectively, where the quasi-derivatives \eqref{quasiz} are used. The following relations, similar to \eqref{PhiCM}, hold:
$$
\Phi_k^{\star}(x, \la) = C_k^{\star}(x, \la) + \sum_{j = k+1}^3 M_{j,k}^{\star}(\la) C_j^{\star}(x, \la), \quad k = 1, 2, 3,
$$
where $M_{j,k}^{\star}(\la)$ are the entries of the Weyl-Yurko matrix $M^{\star}(\la)$ analogous to $M(\la)$.

The relationship between $M(\la)$ and $M^{\star}(\la)$ has been established in Section~2 of \cite{Bond22-alg}. In particular, for the third-order case, Lemma~2 of \cite{Bond22-alg} implies the following proposition. 

\begin{prop} \label{prop:M*}
$M_{2,1}(\la) \equiv M^{\star}_{3,2}(\la)$, $M_{3,2}(\la) \equiv M_{2,1}^{\star}(\la)$, and $M_{3,1}^{\star}(\la) - M_{2,1}^{\star} M_{2,1}(\la) + M_{3,1}(\la) \equiv 0$.   
\end{prop}

\begin{cor} \label{cor:M*}
 The function $M_{3,2}^{\star}(\la)$ has the simple poles at $\la = \la_{n,1}$ such that $\be_{n,1} \ne 0$, and the function $M_{2,1}^{\star}(\la)$ has the simple poles at $\la = \la_{n,2}$ such that $\be_{n,2} \ne 0$. Moreover,
\begin{align*}
& \be_{n,1} = -\Res_{\la = \la_{n,1}} M_{3,2}^{\star}(\la), \quad
\be_{n,2} = -\Res_{\la = \la_{n,2}} M_{2,1}^{\star}(\la), \\
\text{if} \:\: \be_{n,1} = 0 \colon \quad & \ga_n = -\Res_{\la = \la_n} M_{3,1}(\la) = \Res_{\la = \la_n} (M_{3,1}^{\star}(\la) - M_{2,1}^{\star} M_{2,1}(\la)), \\
\text{if} \:\: \be_{n,2} = 0 \colon \quad & \ga_n = -\Res_{\la = \la_n} (M_{3,1}(\la) - M_{2,1}^{\star}(\la) M_{2,1}(\la)) = \Res_{\la = \la_n} M_{3,1}^{\star}(\la).
\end{align*}
\end{cor}

Consequently, the functions $\Phi_1^{\star}(x, \la)$ and $\Phi_2^{\star}(x, \la)$ are meromorphic in $\la$ with the simple poles $\{\la_{n,2} \}_{n = 1}^{\infty}$ and $\{\la_{n,1} \}_{n = 1}^{\infty}$, respectively, and $\Phi_3^{\star}(x, \la)$ is entire in $\la$. 

\section{Main equation} \label{sec:equ}

In this section, we provide the main equation of Inverse Problem~\ref{ip:main}. The main equation for higher-order differential operators with distribution coefficients in the general form has been derived in \cite{Bond22-alg}. Here, we introduce necessary notations and adapt the results of \cite{Bond22-alg} to our case.

Consider equation \eqref{eqv} with the coefficients $\mathcal T = (\tau_0, \tau_1) \in W$ and the spectral data $\mathfrak S$. Choose any coefficients $\tilde {\mathcal T} = (\tilde \tau_0, \tilde \tau_1)$ satisfying the following conditions:
\begin{enumerate}
    \item $\int\limits_0^1 \tau_1(x) \, dx = \int\limits_0^1 \tilde \tau_1(x) \, dx$.
    \item $\tilde {\mathcal T} \in W$.
    \item $\{ \tilde \la_{n,1} \} \cap \{ \tilde \la_{n,2} \} = \varnothing$, that is, $\tilde K = \varnothing$.
    \item $\tilde \la_{n,k} \ne \la_{n_0,k_0}$, $n, n_0 \in \mathbb N$, $k,k_0 = 1, 2$.
\end{enumerate}

We call $\tilde {\mathcal T}$ the model problem.
It can be shown that, for any values $\int_0^1 \tau_1(x) \, dx$ and $\{ \la_{n,k} \}_{n \in \mathbb N, \, k = 1, 2}$ satisfying the conditions 1 and 2 of Theorem~\ref{thm:nsc}, a model problem with the above properties exists.
Indeed, in order to achieve the condition 1, one can take $\tilde \tau_1 := \int\limits_0^1 \tau_1(x) \, dx$ and $\tilde \tau_0 := 0$. If the conditions 2--4 are not fulfilled, one can implement a minor shift of the spectral data to achieve these conditions. On the other hand, the conditions 2--4 are not principal. If some of them do not hold, the main results still remain valid, but the form of the main equation will be slightly different.

Note that the problems with the coefficients $\mathcal T$ and $\tilde{\mathcal T}$ have different quasi-derivatives. Recall that the quasi-derivatives related to $\mathcal T$ are defined via \eqref{quasi} by using the entries of the associated matrix $F(x)$ given by \eqref{defF}. The quasi-derivatives related to $\tilde {\mathcal T}$ are generated similarly by the following associated matrix:
$$
\tilde F(x) = \begin{bmatrix}
            0 & 1 & 0 \\
            -(\tilde \sigma_0 + \tilde \tau_1) & 0 & 1 \\
            0 & (\tilde \sigma_0 - \tilde \tau_1) & 0
        \end{bmatrix}, \qquad (\tilde \tau_0 = \tilde \sigma_0').
$$

Analogously to $F^{\star}(x)$ (see \eqref{defF*}), define the matrix function $\tilde F^{\star}(x)$, the corresponding quasi-derivatives, etc.
For $y \in \mathcal D_{\tilde F}$ and $z \in \mathcal D_{\tilde F^{\star}}$, define the Lagrange bracket
$$
\langle z, y \rangle := z^{[2]} y - z' y' + z y^{[2]},
$$
where $z^{[2]} = z'' + (\tilde \sigma_0 + \tilde \tau_1) z$, $y^{[2]} = y'' - (\tilde \sigma_0 - \tilde \tau_1) y$. If $y$ and $z$ satisfy the relations $\tilde \ell(y) = \mu y$ and $\tilde \ell^{\star}(z) = \la z$, respectively, then
\begin{equation} \label{wron}
\frac{d}{dx} \langle z, y \rangle = (\mu - \la) z y.
\end{equation}

Denote
\begin{align} \label{defD}
\tilde D_{k,j}(x, \la, \mu) = \frac{\langle \tilde \Phi_k^{\star}(x, \la), \tilde \Phi_j(x, \mu)\rangle}{\mu - \la}, \quad k,j = 1, 2, 3, \\ \nonumber
\tilde D_{k,j,\langle 0 \rangle}(x, \la, \mu_0) = (\tilde D_{k,j}(x, \la, \mu))_{\langle 0 \rangle, \, \mu = \mu_0}.
\end{align}

Using \eqref{wron} and the initial conditions on $\tilde \Phi^{\star}_k$ and $\tilde \Phi_j$ at $x = 0$, we obtain
\begin{align} \label{D32}
& \tilde D_{k,j}(x, \la, \mu) = \int_0^x \tilde \Phi_k^{\star}(t, \la) \tilde \Phi_j(t, \mu) \, dt, \quad (k,j) = (2,3), (3,2), (3,3), \\ \label{D22}
& \tilde D_{2,2}(x, \la, \mu) = \frac{1}{\la - \mu} + \int_0^x \tilde \Phi_2^{\star}(t, \la) \tilde \Phi_2(t, \mu) \, dt.
\end{align}

Introduce the notations
\begin{gather} \nonumber
    V := \{ (n,k,\eps) \colon n \in \mathbb N, \, k = 1, 2, \, \eps = 0, 1 \},  \\ \nonumber
    \la_{n,k,0} := \la_{n,k}, \quad \la_{n,k,1} := \tilde \la_{n,k}, \quad \be_{n,k,0} := \be_{n,k}, \quad \be_{n,k,1} := \tilde \be_{n,k}, \\ \nonumber
    \vv_{n,k,\eps}(x) := \Phi_{k+1, \langle 0 \rangle}(x, \la_{n,k,\eps}), \quad \tilde \vv_{n,k,\eps}(x) := \tilde \Phi_{k+1}(x, \la_{n,k,\eps}), \quad (n,k,\eps) \in V, \\ \label{defG}
    \tilde G_{(n,k,\eps), (n_0,k_0,\eps_0)}(x) := 
    \begin{cases}
        (\be_{n,2,0} \tilde D_{2,k_0+1,\langle 0 \rangle} - \ga_n \tilde D_{3,k_0+1})(x, \la_{n,2,0}, \la_{n_0, k_0, \eps_0}), \:\: \text{if} \:\: n \in K, \, k = 2, \, \eps = 0, \\
        (-1)^k \beta_{n,k,\eps} \tilde D_{4-k, k_0+1}(x, \la_{n,k,\eps}, \la_{n_0,k_0,\eps_0}), \quad \text{otherwise}.
    \end{cases}
\end{gather}

Note that the functions $\Phi_2(x, \la)$ and $\tilde \Phi_2(x, \la)$ have the poles $\{ \la_{n,2,0} \}$ and $\{ \la_{n,2,1} \}$, respectively, and the functions $\Phi_3(x, \la)$ and $\tilde \Phi_3(x, \la)$ are entire, so $\langle 0 \rangle$ is absent in the definition of $\tilde \vv_{n,k,\eps}(x)$ and can be removed in the definition of $\varphi_{n,k,\eps}(x)$ in all the cases except $n \in K$, $k = 1$, $\eps = 0$. Analogously, in view of \eqref{D32} and \eqref{D22}, 
$$
\tilde D_{k,j}(x, \la, \mu) \:\: \textit{is analytic for} \:\: 
\begin{cases}
\la \ne \la_{n,1,1}, \, \mu \ne \la_{n,2,1}, \, \la \ne \mu, \quad (k,j) = (2,2), \\
\la \ne \la_{n,1,1}, \quad (k,j) = (2,3), \\
\mu \ne \la_{n,2,1}, \quad (k,j) = (3,2), \\
\text{all $\la$ and $\mu$}, \quad (k,j) = (3,3).
\end{cases}
$$

Hence, the functions $\tilde G_{(n,k,\eps),(n_0,k_0,\eps_0)}(x)$ are correctly defined for $(n,k,\eps), (n_0,k_0,\eps_0) \in V$ and $\langle 0 \rangle$ is necessary only for $n \in K$, $(n,k,\eps) = (n,2,0)$, $(n_0,k_0,\eps_0) = (n,1,0)$.

The following proposition has been proved in \cite{Bond22-alg} by the contour integral method.

\begin{prop}[\cite{Bond22-alg}]
For $(n_0,k_0,\eps_0) \in V$, $x \in [0,1]$, the following relation holds:
\begin{equation} \label{infphi}
\vv_{n_0, k_0,\eps_0}(x) = \tilde \vv_{n_0,k_0,\eps_0}(x) + \sum_{(n,k,\eps) \in V}(-1)^{\eps} \vv_{n,k,\eps}(x) \tilde G_{(n,k,\eps), (n_0, k_0,\eps_0)}(x).
\end{equation}
\end{prop}

The relations \eqref{infphi} can be treated as an infinite linear system of equations with respect to $\vv_{n,k,\eps}(x)$, $(n,k,\eps) \in V$. The elements $\tilde \vv_{n_0,k_0,\eps_0}(x)$ and $G_{(n,k,\eps), (n_0, k_0,\eps_0)}(x)$ are constructed by using the model problem $\tilde {\mathcal T}$ and the spectral data $\mathfrak S$ and $\tilde{\mathfrak S}$ of the both problems, while the elements $\vv_{n,k,\eps}(x)$ are related to the desired coefficients $\mathcal T$. Thus, the system \eqref{infphi} can be used for solving Inverse Problem~\ref{ip:main}. However, it is inconvenient to use \eqref{infphi} as the main equations of the inverse problem, because the series in \eqref{infphi} converges only ``with brackets'':
$$
\sum_{(n,k,\eps) \in V} = \sum_{(n,k)} \left( \sum_{\eps = 0, 1} (\dots)\right).
$$

Therefore, in \cite{Bond22-alg}, the system \eqref{infphi} has been transformed to a linear equation in the Banach space $m$ of bounded infinite sequences. Let us provide that transform. 

Define the numbers $\{ \xi_n \}$ which characterize ``the difference'' of the spectral data $\mathfrak S$ and $\tilde {\mathfrak S}$:
\begin{equation} \label{defxi}
\xi_n := \sum_{k = 1, 2} (n^{-2}|\la_{n,k} - \tilde \la_{n,k}| + n^{-3}|\beta_{n,k} - \tilde \be_{n,k}|), \quad n \in \mathbb N.
\end{equation}

It follows from the asymptotics \eqref{asymptla} that $\{ n \xi_n \} \in l_2$. By Lemma~7 of \cite{Bond22-alg}, the following estimates hold:
\begin{equation} \label{estvv}
|\vv_{n,k,\eps}(x)| \le C w_{n,k}(x), \quad |\vv_{n,k,0}(x) - \vv_{n,k,1}(x)| \le C w_{n,k}(x) \xi_n, 
\end{equation}
where 
$$
w_{n,k}(x) := n^{-k} \exp(-xl \cot (k\pi/n)),
$$
and the constant $C$ does not depend on $x$, $n$, $k$, and $\eps$. 

Introduce the notations
\begin{equation} \label{defpsi}
\begin{bmatrix}
\psi_{n,k,0}(x) \\ \psi_{n,k,1}(x)
\end{bmatrix} := 
w_{n,k}^{-1}(x)
\begin{bmatrix}
\xi_n^{-1} & -\xi_n^{-1} \\ 0 & 1
\end{bmatrix}
\begin{bmatrix}
\vv_{n,k,0}(x) \\ \vv_{n,k,1}(x)
\end{bmatrix}, 
\end{equation}
\begin{multline} \label{defR}
\begin{bmatrix}
\tilde R_{(n_0,k_0,0),(n,k,0)}(x) & \tilde R_{(n_0,k_0,0),(n,k,1)}(x) \\
\tilde R_{(n_0,k_0,1),(n,k,0)}(x) & \tilde R_{(n_0,k_0,1),(n,k,1)}(x)
\end{bmatrix} := \\ 
\frac{w_{n,k}(x)}{w_{n_0,k_0}(x)}
\begin{bmatrix}
\xi_{n_0}^{-1} & -\xi_{n_0}^{-1} \\ 0 & 1
\end{bmatrix}
\begin{bmatrix}
\tilde G_{(n,k,0),(n_0,k_0,0)}(x) & \tilde G_{(n,k,1),(n_0,k_0,0)}(x) \\
\tilde G_{(n,k,0),(n_0,k_0,1)}(x) & \tilde G_{(n,k,1),(n_0,k_0,1)}(x)
\end{bmatrix}
\begin{bmatrix}
\xi_n & 1 \\ 0 & -1
\end{bmatrix},
\end{multline}
and define $\tilde \psi_{n,k,\eps}(x)$ analogously to $\psi_{n,k,\eps}(x)$.

For brevity, put $\psi_v(x) := \psi_{n,k,\eps}(x)$, $\tilde \psi_v(x) := \tilde \psi_{n,k,\eps}(x)$, $\tilde R_{v_0,v}(x) := \tilde R_{(n_0,k_0,\eps_0),(n,k,\eps)}(x)$, $v = (n,k,\eps)$, $v_0 = (n_0,k_0,\eps_0)$, $v,v_0 \in V$. Then, the relations \eqref{infphi} can be rewritten as follows:
\begin{equation} \label{sumpsi}
    \psi_{v_0}(x) = \tilde \psi_{v_0}(x) + \sum_{v \in V} \tilde R_{v_0,v}(x) \psi_v(x), \quad v_0 \in V
\end{equation}

In \cite{Bond22-alg}, the following estimates have been obtained:
\begin{equation} \label{estpsiR}
|\psi_v(x)|, |\tilde \psi_v(x)| \le C, \quad
|\tilde R_{v_0,v}(x)| \le \frac{C\xi_n}{|n-n_0| + 1}, \quad v,v_0 \in V,
\end{equation}

It follows from \eqref{estpsiR} that the series in \eqref{sumpsi} converges absolutely and uniformly with respect to $x \in [0,1]$.

Consider the Banach space $m$ of bounded infinite sequences $\al = [\al_v]_{v \in V}$ with the norm $\| \al \|_m = \sum\limits_{v \in V} |\al_v|$. Obviously, $\psi(x), \tilde \psi(x) \in m$ for each fixed $x \in [0,1]$. Define the linear operator $\tilde R(x) = [\tilde R_{v_0,v}(x)]_{v_0, v \in V}$ acting on an element $\al = [\al_v]_{v \in V} \in m$ by the following rule:
\begin{equation} \label{Ral}
[\tilde R(x) \al]_{v_0} = \sum_{v \in V} \tilde R_{v_0,v}(x) \al_v, \quad v_0 \in V.
\end{equation}

The main result of this section is formulated as follows.

\begin{prop}[\cite{Bond22-alg}] \label{prop:maineq}
For each fixed $x\in[0,1]$, the linear operator $\tilde R(x)$ is compact in $m$ and can be approximated by finite-rank operators. Furthermore, the following relation holds
\begin{equation} \label{main}
    (I - \tilde R(x)) \psi(x) = \tilde \psi(x), \quad x \in [0,1],
\end{equation}
where $I$ is the unit operator in $m$.
\end{prop}

The relation \eqref{main} is called the main equation of Inverse Problem~\ref{ip:main}. It plays an important role in the proofs of Theorems~\ref{thm:nsc}, \ref{thm:loc}, and~\ref{thm:suff} in the next sections.

\section{Proofs of Theorems~\ref{thm:nsc} and~\ref{thm:loc}} \label{sec:proof}

In this section, we prove Theorem~\ref{thm:nsc} on the necessary and sufficient conditions of the inverse problem solvability. Then, as a corollary, we obtain Theorem~\ref{thm:loc} on the local solvability and stability.

Suppose that $\mathfrak S = (\{ \la_{n,k} \}_{n \in \mathbb N, \, k = 1, 2}, \{ \beta_{n,k} \}_{n \in \mathbb N, \, k = 1, 2}, \{ \ga_n \}_{n \in K})$ are the spectral data of the corresponding boundary value problems for equation \eqref{eqv} with the coefficients $\mathcal T = (\tau_0,\tau_1) \in W$. The necessity of the conditions 1 and 2 in Theorem~\ref{thm:nsc} (asymptotics and structural properties of the spectral data) is given by Lemmas~\ref{lem:weight} and~\ref{lem:asympt}. The invertibility of the operator $(I - \tilde R(x))$ from the main equation \eqref{main} has been proved in \cite{Bond22-alg}. The inverse operator $(I - \tilde R(x))^{-1}$ has been found explicitly in the case when the differential expression coefficients are known (see Theorem~1 in \cite{Bond22-alg}). Thus, the necessity part of Theorem~\ref{thm:nsc} has been already proved. Therefore, in this section, we focus on the proof of the sufficiency.

Let $\mathfrak S$ be any numbers satisfying the conditions 1--3 of Theorem~\ref{thm:nsc}. We have to prove the existence of some coefficients $\mathcal T = (\tau_0, \tau_1)$ such that the numbers $\mathfrak S$ are their spectral data. Let us outline the proof.

\begin{enumerate}
    \item The solution $\psi(x)$ of the main equation \eqref{main} is constructed and its properties are studied.
    \item Using the entries of $\psi(x)$, we find the functions $\tau_0$ and $\tau_1$ by the reconstruction formulas \eqref{rec1} and \eqref{rec0}.
    \item We prove that $\tau_0 \in W_2^{-1}(0,1)$ and $\tau_1 \in L_2(0,1)$ (Lemma~\ref{lem:series}).
    \item We prove that the initially given numbers $\mathfrak S$ are the spectral data of equation \eqref{eqv} with the constructed coefficients $\tau_0$ and $\tau_1$ (Lemma~\ref{lem:sdt}). For this purpose, the approximation approach is used.
    \begin{enumerate}
        \item We define the ``truncated'' spectral data $\mathfrak S^N$ by \eqref{defSN} and \eqref{deflaN} and prove the unique solvability of the corresponding main equation, which turns into a finite linear system (Lemma~\ref{lem:solveN}).
        \item We construct the functions $\tau_1^N$ and $\tau_0^N$ analogous to $\tau_1$ and $\tau_0$, respectively, and the functions $\Phi^N_k(x, \la)$, $k = 1, 2, 3$, by the formula \eqref{defPhiN}. Then, we prove that $\Phi^N_k(x, \la)$ are the Weyl solutions and $\mathfrak S^N$ are the spectral data of $\mathcal T^N = (\tau_0^N, \tau_1^N)$ (Lemmas~\ref{lem:PhiNeq} and~\ref{lem:PhiNbc}). The advantage of considering the ``truncated'' data $\mathfrak S^N$ is that the series for $\Phi^N_k(x, \la)$ are finite and so these functions can be easily substituted into equation \eqref{eqv} with the coefficients $\mathcal T^N$.
        \item It is shown that $\tau_1^N \to \tau_1$ and $\tau_0^N \to \tau_0$ in the corresponding spaces as $N \to \infty$ (Lemma~\ref{lem:limtau}).
        \item We prove the stability of the spectral data with respect to $\tau_1$ and $\tau_0$ (Lemma~\ref{lem:limsd}).
        \item The fact that $\mathfrak S^N \to \mathfrak S$ finishes the proof of Lemma~\ref{lem:sdt}.
    \end{enumerate}
\end{enumerate}

Now, proceed to the detailed proof. By virtue of the condition~3, the operator $(I - \tilde R(x))$ has a bounded inverse. Therefore, the main equation \eqref{main} has a unique solution in $m$ for each fixed $x \in [0,1]$:
$$
\psi(x) = (I - \tilde R(x))^{-1} \tilde \psi(x), \quad \psi(x) = [\psi_v(x)]_{v \in V}.
$$

Note that the estimates \eqref{estpsiR} on the entries $\tilde \psi_v(x)$ and $\tilde R_{v_0, v}(x)$ are similar to the ones for the Sturm-Liouville operator with regular potential $q \in L_2$ (see, e.g., \cite[Section~1.6.1]{FY01}). Therefore, repeating the proof arguments of Lemma~1.6.7 in \cite{FY01}, we easily show that the entries $\psi_v(x)$ of the main equation solution have the following properties:
\begin{equation} \label{proppsi}
\psi_v(x) \in C^1[0,1], \quad |\psi^{(\nu)}(x)| \le C n^{\nu}, \quad 
|\psi_v(x) - \tilde \psi_v(x)| \le C \chi_n, \quad |\psi_v'(x) - \tilde \psi_v'(x)| \le C, 
\end{equation}
where $v = (n,k,\eps) \in V$, $x \in [0,1]$, $\nu = 0, 1$, and
$$
\chi_n := \left( \sum_{k = 1}^{\infty} \frac{1}{k^2 (|n-k|+1)^2} \right)^{1/2}, \quad \{ \chi_n \} \in l_2.
$$

Using the entries $\psi_v(x)$, construct the functions $\vv_{n,k,\eps}(x)$ by inverting the formula \eqref{defpsi}:
\begin{equation} \label{defvvpsi}
\begin{bmatrix}
    \vv_{n,k,0}(x) \\ \vv_{n,k,1}(x)
\end{bmatrix} = w_{n,k}(x)
\begin{bmatrix}
    \xi_n & 1 \\ 0 & 1
\end{bmatrix}
\begin{bmatrix}
    \psi_{n,k,0}(x) \\ \psi_{n,k,1}(x).
\end{bmatrix}
\end{equation}

It follows from \eqref{proppsi} that
\begin{gather} \nonumber
    \vv_{n,k,\eps}(x) \in C^1[0,1], \quad |\vv^{(\nu)}_{n,k,\eps}(x)| \le C w_{n,k}(x) n^{\nu}, \quad
    |\vv^{(\nu)}_{n,k,0}(x) - \vv^{(\nu)}_{n,k,1}(x)| \le C w_{n,k}(x) n^{\nu} \xi_n, \\ \label{estvv1}
\left.\begin{array}{c}
    |\vv_{n,k,\eps}^{(\nu)}(x) - \tilde \vv^{(\nu)}_{n,k,\eps}(x)| \le C w_{n,k}(x) \chi_n^{1 - \nu}, \\
    |\vv_{n,k,0}^{(\nu)}(x) - \vv_{n,k,1}^{(\nu)}(x) - \tilde \vv_{n,k,0}^{(\nu)}(x) + \tilde \vv_{n,k,1}^{(\nu)}(x)| \le C w_{n,k}(x) \xi_n \chi_n^{1 - \nu},
\end{array}\right\}
\end{gather}
for $(n,k,\eps) \in V$, $\nu = 0, 1$, $x \in [0,1]$.

Using the functions $\vv_{n,k,\eps}(x)$, find $\tau_1$ and $\tau_0$ by the reconstruction formulas (see \cite[Section~4.3]{Bond22-alg}):
\begin{align} \label{rec1}
    & \tau_1 := \tilde \tau_1 - \frac{3}{2} \sum_{V} (-1)^{\eps}(\vv'_{n,k,\eps} \tilde \eta_{n,k,\eps} + \vv_{n,k,\eps} \tilde \eta'_{n,k,\eps}), \\ \label{rec0}
    & \tau_0 := \tilde \tau_0 - \hat \tau_1' - 3 \frac{d}{dx}  \sum_{V}(-1)^{\eps} \vv_{n,k,\eps}' \tilde \eta_{n,k,\eps} - 2 \hat \tau_1 \sum_{V} (-1)^{\eps}\vv_{n,k,\eps} \tilde \eta_{n,k,\eps},
\end{align}
where $\hat \tau_1 = \tau_1 - \tilde \tau_1$,
\begin{equation} \label{defeta}
\eta_{n,k,\eps}(x) := 
\begin{cases}
(\beta_{n,2,0} \tilde \Phi_2^{\star} - \ga_n \tilde \Phi_3^{\star}) (x, \la_{n,2,0}), \quad n \in K, \, k = 2, \, \eps = 0, \\
(-1)^k \beta_{n,k,\eps} \tilde \Phi_{4-k}^{\star}(x, \la_{n,k,\eps}), \quad \text{otherwise}.
\end{cases} 
\end{equation}

The functions $\tilde \eta_{n,k,\eps}(x)$ are defined by using the model problem and the spectral data $\mathfrak S$, so they can be estimated similarly to \eqref{estvv}:
\begin{equation} \label{esteta}
\tilde \eta_{n,k,\eps} \in C^1[0,1], \quad
|\tilde \eta_{n,k,\eps}^{(\nu)}(x)| \le C w_{n,k}^{-1}(x) n^{\nu}, \quad
|\tilde \eta_{n,k,0}^{(\nu)}(x) - \tilde \eta_{n,k,1}^{(\nu)}(x)| \le C w_{n,k}^{-1}(x) n^{\nu} \xi_n,
\end{equation}
for $(n,k,\eps) \in V$, $\nu = 0, 1$, $x \in [0,1]$.

\begin{lem} \label{lem:series}
The series \eqref{rec1} converges in $L_2(0,1)$ and the formula \eqref{rec0} defines the function of $W_2^{-1}(0,1)$.
\end{lem}

\begin{proof}
Represent the series in \eqref{rec1} as the sum of the two series
\begin{align*}
    & \mathscr S_1 := \sum_V (-1)^{\eps}((\vv'_{n,k,\eps} - \tilde \vv'_{n,k,\eps}) \tilde \eta_{n,k,\eps} + (\vv_{n,k,\eps} - \tilde \vv_{n,k,\eps}) \tilde \eta'_{n,k,\eps}), \\ 
    & \mathscr S_2 := \sum_V (-1)^{\eps}(\tilde \vv'_{n,k,\eps} \tilde \eta_{n,k,\eps} + \tilde \vv_{n,k,\eps} \tilde \eta'_{n,k,\eps}).
\end{align*}
Obviously,
\begin{align*}
    \mathscr S_1 = & \sum_{(n,k)} \bigl((\vv'_{n,k,0} - \vv'_{n,k,1} - \tilde \vv'_{n,k,0} + \tilde \vv'_{n,k,1}) \tilde \eta_{n,k,0} + (\vv'_{n,k,1} - \tilde \vv'_{n,k,1}) (\tilde \eta_{n,k,0} - \tilde \eta_{n,k,1}) \\
    & + (\vv_{n,k,0} - \vv_{n,k,1} - \tilde \vv_{n,k,0} + \tilde \vv_{n,k,1}) \tilde \eta'_{n,k,0} + (\vv_{n,k,1} - \tilde \vv_{n,k,1}) (\tilde \eta'_{n,k,0} - \tilde \eta'_{n,k,1})\bigr).
\end{align*}

Using the estimates \eqref{estvv1} and \eqref{esteta}, we obtain
\begin{equation} \label{estS1}
|\mathscr S_1| \le C \sum_{n = 1}^{\infty} \xi_n + C \sum_{n = 1}^{\infty} n \xi_n \chi_n.
\end{equation}

Recall that $\{ n \xi_n \} \in l_2$ and $\{ \chi_n \} \in l_2$. Hence, the both series in the right-hand side of \eqref{estS1} converge, so the series $\mathscr S_1$ converges absolutely and uniformly with respect to $x \in [0,1]$. For the series $\mathscr S_2$, the convergence in $L_2(0,1)$ follows from Lemma~8 in \cite{Bond22-alg}. Therefore, the series for $\tau_1(x)$ in \eqref{rec1} converges in $L_2(0,1)$. 

Thus $\hat \tau_1' \in W_2^{-1}(0,1)$. Let us study the convergence of the other terms in \eqref{rec0}. The series $\sum\limits_V \vv'_{n,k,\eps} \tilde \eta_{n,k,\eps}$ can be formally represented as $(\mathscr T_1 + \mathscr T_2)$, where
$$
\mathscr T_1 := \sum_V (\vv_{n,k,\eps}' - \tilde \vv_{n,k,\eps}') \tilde \eta_{n,k,\eps}, \quad
\mathscr T_2 := \sum_V \tilde \vv'_{n,k,\eps} \tilde \eta_{n,k,\eps}.
$$

The series $\mathscr T_1$, similarly to $\mathscr S_1$, converges absolutely and uniformly with respect to $x \in [0,1]$. By virtue of Lemma~8 from \cite{Bond22-alg}, the series $\mathscr T_2$ converges in $L_2(0,1)$ with some regularization constants $a_{n,k,\eps}$:
$$
\sum_V (\tilde \vv'_{n,k,\eps}(x) \tilde \eta_{n,k,\eps}(x) - a_{n,k,\eps}) \in L_2(0,1).
$$

Anyway, the constants $a_{n,k,\eps}$ are unimportant because of the differentiation in \eqref{rec0}. The series $\sum\limits_V (-1)^{\eps} \vv_{n,k,\eps} \tilde \eta_{n,k,\eps}$ converges absolutely and uniformly on $[0,1]$. Hence, the function $\tau_0$ defined by \eqref{rec0} belongs to $W_2^{-1}(0,1)$.
\end{proof}

Now consider equation \eqref{eqv} with the coefficients $\mathcal T = (\tau_0, \tau_1)$ defined via \eqref{rec1}--\eqref{rec0} by using the initially given data $\mathfrak S$. In order to finish the proof Theorem~\ref{thm:nsc}, it remains to prove the following lemma.

\begin{lem} \label{lem:sdt}
$\mathfrak S$ are the spectral data of $\mathcal T = (\tau_0, \tau_1)$.
\end{lem}

For sufficiently large $N \in \mathbb N$ ($N \ge \max K$), define the data
\begin{gather} \label{defSN}
    \mathfrak S^N := (\{ \la_{n,k}^N \}_{n \in \mathbb N, \, k = 1, 2}, \{ \be_{n,k}^N \}_{n \in \mathbb N, \, k = 1, 2}, \{ \ga_n \}_{n \in K}), \\ \label{deflaN}
    \la_{n,k}^N := \begin{cases}
                        \la_{n,k}, \quad n \le N, \\
                        \tilde \la_{n,k}, \quad n > N,
                    \end{cases}
    \qquad
    \be_{n,k}^N := \begin{cases}
                        \be_{n,k}, \quad n \le N, \\
                        \tilde \be_{n,k}, \quad n > N.
                    \end{cases}    
\end{gather}

Repeat the arguments of Section~\ref{sec:equ} using the same model problem $\tilde {\mathcal T}$ and the data $\mathfrak S^N$ instead of $\mathfrak S$. Then, instead of the infinite system~\eqref{sumpsi}, we obtain the finite linear system
\begin{equation} \label{mainN}
\psi_{v_0}^N(x) = \tilde \psi_{v_0}(x) + \sum_{v \in V^N} \tilde R_{v_0,v}(x) \psi_v^N(x), \quad v_0 \in V^N, \quad x \in [0,1],
\end{equation}
where
$$
V^N := \{ (n,k,\eps) \in V \colon n \le N \}.
$$

\begin{lem} \label{lem:solveN}
For every sufficiently large $N$, the system \eqref{mainN} is uniquely solvable for each $x \in [0,1]$.
\end{lem}

\begin{proof}
Consider an auxiliary operator
\begin{align} \nonumber
\tilde R^N(x) & = [\tilde R_{v_0,v}^N(x)]_{v_0,v \in V} \colon m \to m, \\
\label{defRN}
\tilde R^N_{(n_0,k_0,\eps_0),(n,k,\eps)}(x)  & = 
\begin{cases}
    0, \quad n > N, \\
    \tilde R_{(n_0,k_0,\eps_0),(n,k,\eps)}(x), \quad \text{otherwise}.
\end{cases}
\end{align}

Clearly, the operator $\tilde R^N(x)$ is finite-rank, and the operator $(I - \tilde R^N(x))$ has a bounded inverse if and only if the matrix of the system \eqref{mainN} is non-singular.

Using \eqref{estpsiR} and \eqref{defRN}, we estimate
$$
\| \tilde R(x) - \tilde R^N(x) \|_{m \to m} \le \sup_{n_0 \ge 1} \sum_{n = N+1}^{\infty} \frac{C \xi_n}{|n - n_0| + 1} \le \frac{C}{N} \sqrt{\sum_{n = N+1}^{\infty} (n \xi_n)^2}.
$$
Recall that $\{ n \xi_n \} \in l_2$. Hence 
$$
\| \tilde R(x) - \tilde R^N(x) \|_{m \to m} = o\left(N^{-1}\right), \quad N \to \infty.
$$

Since the operator $(I - \tilde R(x))$ has a bounded inverse, then so does $(I - \tilde R^N(x))$ for sufficiently large $N$. Consequently, equation \eqref{mainN} is uniquely solvable.
\end{proof}

The proof of Lemma~\ref{lem:solveN} readily implies the following corollary.

\begin{cor} \label{cor:RN}
The solution $\{ \psi_v^N(x) \}_{v \in V^N}$ of the finite system \eqref{mainN} can be found as the entries with the indices $v \in V^N$ of the solution $\psi^N(x) \in m$ of the equation
$$
(I - \tilde R^N(x)) \psi^N(x) = \tilde \psi(x), \quad x \in [0,1],
$$ 
where the operator $\tilde R^N(x)$ is defined by \eqref{defRN}. Moreover, the following estimate holds:
$$
\| R(x) - \tilde R^N(x) \|_{m \to m} \le \eps_N, \quad \eps_N = o\left(N^{-1}\right), \quad N \to \infty.
$$
\end{cor}

Using the entries of the solution $\psi^N(x)$, define the functions
\begin{align} \label{defvvpsiN}
& \begin{bmatrix}
    \vv^N_{n,k,0}(x) \\ \vv^N_{n,k,1}(x)
\end{bmatrix}  := w_{n,k}(x)
\begin{bmatrix}
    \xi_n & 1 \\ 0 & 1
\end{bmatrix}
\begin{bmatrix}
    \psi^N_{n,k,0}(x) \\ \psi^N_{n,k,1}(x),
\end{bmatrix}
\quad n \le N, \\ \label{rec1N}
  &  \tau_1^N := \tilde \tau_1 - \frac{3}{2} \sum_{V^N} (-1)^{\eps}((\vv^N_{n,k,\eps})' \tilde \eta_{n,k,\eps} + \vv^N_{n,k,\eps} \tilde \eta'_{n,k,\eps}), \quad \hat \tau_1^N := \tau_1^N - \tilde \tau_1, \\ \label{rec0N}
  &  \tau_0^N := \tilde \tau_0 - (\hat \tau_1^N)' - 3 \frac{d}{dx}  \sum_{V^N}(-1)^{\eps} (\vv_{n,k,\eps}^N)' \tilde \eta_{n,k,\eps} - 2 \hat \tau_1^N \sum_{V^N} (-1)^{\eps}\vv_{n,k,\eps}^N \tilde \eta_{n,k,\eps},
\end{align}
analogously to the formulas \eqref{defvvpsi}, \eqref{rec1}, and \eqref{rec0}, respectively. Consider the following equation, the associated matrix, the corresponding quasi-derivatives, and the domain:
\begin{gather} \label{eqvN}
y''' + (\tau_1^N y)' + (\tau_1^N)' y + \tau_0^N y = \la y, \quad x \in (0,1), \\ \nonumber
F^N(x) := \begin{bmatrix}
                0 & 1 & 0 \\
                -(\sigma_0^N + \tau_1^N) & 0 & 1 \\
                0 & (\sigma_0^N - \tau_1^N) & 0
           \end{bmatrix},  \\ \label{quasiN}
y^{[0]} := y, \quad y^{[k]} := (y^{[k-1]})' - \sum_{j = 1}^k f_{k,j}^N y^{[j-1]}, \quad k = 1, 2, 3, \\ \nonumber
\mathcal D_{F^N} := \{ y \colon y^{[k]} \in AC[0,1], \, k = 0, 1, 2 \},
\end{gather}
where
\begin{equation} \label{sigmaN}
\sigma_0^N := \tilde \sigma_0 - \hat \tau_1^N - 3 \sum_{V^N} (-1)^{\eps} (\vv_{n,k,\eps}^N)' \tilde \eta_{n,k,\eps} - 2 \int_0^x \hat \tau_1^N \sum_{V^N} (-1)^{\eps} \vv_{n,k,\eps}^N \tilde \eta_{n,k,\eps} \, dx.
\end{equation}

Furthermore, define the functions
\begin{equation} \label{defPhiN}
\Phi_{k_0}^N(x, \la) := \tilde \Phi_{k_0}(x, \la) + \sum_{(n,k,\eps) \in V^N} (-1)^{\eps} \vv_{n,k,\eps}^N(x) \tilde P_{(n,k,\eps), k_0}(x, \la), \quad k_0 = 1, 2, 3,
\end{equation}
where
\begin{equation} \label{defP}
\tilde P_{(n,k,\eps), k_0}(x, \la) = 
\begin{cases}
(\beta_{n,2,0} \tilde D_{2,k_0} - \ga_n \tilde D_{3,k_0})(x, \la_{n,2,0}, \la), \:\: \text{if} \:\: n \in K, \, k = 2, \, \eps = 0, \\
(-1)^k \beta_{n,k,\eps} \tilde D_{4-k, k_0}(x, \la_{n,k,\eps}, \la), \quad \text{otherwise}.
\end{cases}
\end{equation}

Clearly, 
\begin{equation} \label{relPhiN}
\Phi_{k+1, \langle 0 \rangle}^N(x, \la_{n,k,\eps}) = \vv_{n,k,\eps}^N(x), \quad (n,k,\eps) \in V^N.
\end{equation}

In the next two lemmas, we prove that $\Phi_k^N(x, \la)$, $k = 1, 2, 3$, are the Weyl solutions of equation \eqref{eqvN}.

\begin{lem} \label{lem:PhiNeq}
For sufficiently large $N$, $k = 1, 2, 3$, and $\la \not\in \{ \la_{n,j} \} \cup \{ \tilde \la_{n,j} \}$, we have $\Phi^N_k(., \la) \in \mathcal D_{F^N}$ and $(\Phi_k^N)^{[3]} = \la \Phi_k^N$, where the quasi-derivative is defined by \eqref{quasiN}.    
\end{lem}

\begin{proof}
Recall that $\tilde \Phi_k(., \la) \in \mathcal D_{\tilde F}$, so $\tilde \Phi_k(., \la) \in W_1^2[0,1]$, $k = 1, 2, 3$. Consequently, $\tilde \psi_v \in W_1^2[0,1]$ and $\tilde R_{v_0,v} \in W_1^3[0,1]$ for $v_0, v \in V^N$. Hence, the entries $\tilde \psi_v^N$ of the finite system \eqref{mainN} solution belong to $W_1^2[0,1]$. Therefore, $\vv_{n,k,\eps}^N, \tilde \eta_{n,k,\eps} \in W_1^2[0,1]$ and $\tilde P_{(n,k,\eps),k_0} \in W_1^3[0,1]$ for $(n,k,\eps) \in V^N$, $k_0 = 1, 2, 3$. Using \eqref{rec1N}, \eqref{rec0N}, and \eqref{defPhiN}, we conclude that $\hat \tau_1^N \in W_1^1[0,1]$, $\hat \tau_0^N := \tau_0^N - \tilde \tau_0 \in L_1(0,1)$, and $\Phi_k^N(., \la) \in W_1^2[0,1]$. 

It follows from \eqref{wron}, \eqref{defD}, \eqref{defeta}, and \eqref{defP} that 
$$
\tilde P_{(n,k,\eps),k_0}'(x) = \tilde \eta_{n,k,\eps}(x) \tilde \Phi_{k_0}(x).
$$

Below in this proof, we omit the upper index $N$ and the arguments $(x)$ and $(x, \la)$ for brevity.
Thus, the differentiation of \eqref{defPhiN} implies
\begin{align} \nonumber
    \Phi_{k_0}' & =  \tilde \Phi_{k_0}' + \sum_{V^N} (-1)^{\eps} \left( \vv_{n,k,\eps}' \tilde P_{(n,k,\eps), k_0} + \vv_{n,k,\eps} \tilde \eta_{n,k,\eps} \tilde \Phi_{k_0}\right) \\ \label{PhiN2}
    \Phi_{k_0}'' & =  \tilde \Phi_{k_0}'' + \sum_{V^N} (-1)^{\eps} \Bigl( \vv''_{n,k,\eps} \tilde P_{(n,k,\eps), k_0} + 2 \vv_{n,k,\eps}' \tilde \eta_{n,k,\eps} \tilde \Phi_{k_0} + \vv_{n,k,\eps} (\tilde \eta_{n,k,\eps} \tilde \Phi_{k_0})' \Bigr)
\end{align}

Using the quasi-derivatives given by \eqref{quasiN} for $\Phi_{k_0}$ and $\vv_{n,k,\eps}$, the relations \eqref{sigmaN}, \eqref{defPhiN}, \eqref{PhiN2}, and $\tilde \Phi_{k_0}^{[2]} = \tilde \Phi_{k_0}'' - (\tilde \sigma_0 - \tilde \tau_1) \tilde \Phi_{k_0}$, we obtain
\begin{align} \nonumber
\Phi_{k_0}^{[2]} = & \, \tilde \Phi_{k_0}^{[2]} + \sum_{V^N}(-1)^{\eps} \Bigl( \vv_{n,k,\eps}^{[2]} \tilde P_{(n,k,\eps),k_0} - \vv_{n,k,\eps}' \tilde \eta_{n,k,\eps} \tilde \Phi_{k_0} + \vv_{n,k,\eps} (\tilde \eta_{n,k,\eps} \tilde \Phi_{k_0})'\Bigr) \\ \label{PhiN2q} & - 2 \int_0^x \hat \tau_1 \sum_{V^N} (-1)^{\eps} \vv_{n,k,\eps} \tilde \eta_{n,k,\eps} \, dx \, \tilde \Phi_{k_0}.
\end{align}

For simplicity, suppose that $K = \varnothing$. The opposite case requires minor technical changes. Then \eqref{relPhiN} implies $\Phi_{k_0 + 1}(x, \la_{n_0,k_0,\eps_0}) = \vv_{n_0,k_0,\eps_0}(x)$ and so we get
\begin{equation} \label{sys2}
\vv_{n_0,k_0,\eps_0}^{[2]} = \sum_{V^N} (-1)^{\eps} \vv_{n,k,\eps}^{[2]} \tilde G_{(n,k,\eps),(n_0,k_0,\eps_0)} + r_{n_0,k_0,\eps_0}, \quad (n_0,k_0,\eps_0) \in V^N,
\end{equation}
where $\tilde G_{(n,k,\eps), (n_0,k_0,\eps_0)}$ is defined by \eqref{defG} and
\begin{align*}
r_{n_0,k_0,\eps_0} := & \, \tilde \vv^{[2]}_{n_0,k_0,\eps_0} - \sum_{V^N}(-1)^{\eps} \Bigl( \vv_{n,k,\eps}' \tilde \eta_{n,k,\eps} \tilde \vv_{n_0,k_0,\eps_0} - \vv_{n,k,\eps} (\tilde \eta_{n,k,\eps} \tilde \vv_{n_0,k_0,\eps_0})'\Bigr) \\ & - 2 \int_0^x \hat \tau_1 \sum_{V^N} (-1)^{\eps} \vv_{n,k,\eps} \tilde \eta_{n,k,\eps} \, dx \, \tilde \vv_{n_0,k_0,\eps_0}.
\end{align*}

Since $\tilde \vv_{n_0,k_0,\eps_0} \in \mathcal D_F$, then $\tilde \vv^{[2]}_{n_0,k_0,\eps_0} \in AC[0,1]$, and so $r_{n_0,k_0,\eps_0}\in AC[0,1]$. Hence, the relations \eqref{sys2} can be treated as a finite linear system similar to \eqref{mainN} with the invertible matrix of class $W_1^3[0,1]$ and the right-hand side of class $AC[0,1]$. Therefore, $\vv_{n,k,\eps}^{[2]} \in AC[0,1]$. In view of \eqref{PhiN2q}, this implies $\Phi_{k_0}^{[2]} \in AC[0,1]$. Hence $\Phi_{k_0}^N \in \mathcal D_{F^N}$, $k_0 = 1, 2, 3$.

Thus, we can differentiate \eqref{PhiN2q} and find the third quasi-derivative. Using the relations \eqref{rec1N} and \eqref{rec0N} for simplification, we obtain
\begin{align} \nonumber
\Phi_{k_0}^{[3]} & = \, \tilde \Phi_{k_0}^{[3]} + \sum_{V^N} (-1)^{\eps} \Bigl( \vv_{n,k,\eps}^{[3]} \tilde P_{(n,k,\eps),k_0} \\ \label{PhiN3} & + \vv_{n,k,\eps} \bigl( \tilde \eta_{n,k,\eps}'' \tilde \Phi_{k_0} - \tilde \eta_{n,k,\eps}' \tilde \Phi'_{k_0} + \tilde \eta_{n,k,\eps} \tilde \Phi''_{k_0} + 2 \tilde \tau_1 \tilde \eta_{n,k,\eps} \tilde \Phi_{k_0}\bigr)\Bigr),
\end{align}
where the quasi-derivatives $\Phi_{k_0}^{[3]}$ and $\vv_{n,k,\eps}^{[3]}$ are given by the associated matrix $F^N(x)$ and $\tilde \Phi_{k_0}^{[3]}$ is given by $\tilde F(x)$. Note that
\begin{equation} \label{smeta}
\tilde \eta_{n,k,\eps}'' \tilde \Phi_{k_0} - \tilde \eta_{n,k,\eps}' \tilde \Phi'_{k_0} + \tilde \eta_{n,k,\eps} \tilde \Phi''_{k_0} + 2 \tilde \tau_1 \tilde \eta_{n,k,\eps} \tilde \Phi_{k_0} = \langle \tilde \eta_{n,k,\eps}, \tilde \Phi_{k_0} \rangle = (\la - \la_{n,k,\eps}) \tilde P_{(n,k,\eps), k_0}.
\end{equation}

The relations \eqref{defPhiN}, \eqref{PhiN3}, and \eqref{smeta} together imply
\begin{equation} \label{eqPhiN}
\Phi_{k_0}^{[3]} - \la \Phi_{k_0} = \tilde \Phi_{k_0}^{[3]} - \la \tilde \Phi_{k_0} + \sum_{V^N} (-1)^{\eps} (\vv^{[3]}_{n,k,\eps} - \la_{n,k,\eps}) \tilde P_{(n,k,\eps), k_0}.
\end{equation}

Taking the relation $\tilde \Phi^{[3]}_{k_0} = \la \tilde \Phi_{k_0}$ into account and putting $\la = \la_{n_0,k_0,\eps_0}$, we arrive at the linear algebraic system
$$
(\vv_{n_0,k_0,\eps_0}^{[3]} - \la_{n_0,k_0,\eps_0} \vv_{n_0,k_0,\eps_0}) = \sum_{V^N} (-1)^{\eps} (\vv_{n,k,\eps}^{[3]} - \la_{n,k,\eps} \vv_{n,k,\eps}) \tilde G_{(n,k,\eps),(n_0,k_0,\eps_0)}, \quad (n_0, k_0,\eps_0) \in V^N.
$$

Since this system is homogeneous and its matrix is non-singular, we conclude that $\vv_{n,k,\eps}^{[3]} = \la_{n,k,\eps} \vv_{n,k,\eps}$, $(n,k,\eps) \in V^N$. Consequently, it follows from \eqref{eqPhiN} that $\Phi_{k_0}^{[3]} = \la \Phi_{k_0}$.
\end{proof}

\begin{lem} \label{lem:PhiNbc}
For sufficiently large $N$, the functions $\Phi_k^N(x,\la)$, $k = 1, 2, 3$, fulfill the boundary conditions
\begin{equation} \label{bcPhiN}
(\Phi_k^N)^{(j-1)}(0,\la) = \de_{k,j}, \quad j = \overline{1,k}, \qquad
(\Phi_k^N)^{(3-j)}(1,\la) = 0, \quad j = \overline{k+1,3}.
\end{equation}
Thus, $\Phi_k^N(x,\la)$, $k = 1, 2, 3$, are the Weyl solutions of equation \eqref{eqvN}.
Moreover, $\mathfrak S^N$ are the spectral data of $\mathcal T^N$. \end{lem}

\begin{proof}
Let us focus on the proof of the boundary conditions \eqref{bcPhiN} at $x = 1$, which is more technically complicated. The boundary conditions \eqref{bcPhiN} at $x = 0$ can be checked by direct calculations. For brevity, we omit the upper index $N$. For simplicity, consider the case $\beta_{n,1} = \beta_{n,2} = 0$ for all $n \in K$. The other cases require technical modifications.

\textsc{Step 1}. Let us prove that 
\begin{equation} \label{step1}
\Phi_2^{(\nu)}(1,\la_{n,1,0}) = 0, \:\: n = \overline{1,N} \setminus K, \qquad
\Phi_3^{(\nu)}(1,\la_n) = 0, \:\: n \in K, \quad \nu = 0, 1.
\end{equation}

Using the relations \eqref{defPhiN}, \eqref{relPhiN}, and $\tilde \Phi_2(1,\la) = 0$, we get
\begin{align} \nonumber
\Phi_2(1,\la) = & -\sum_{(n,1,\eps) \in V^N} (-1)^{\eps} \be_{n,1,\eps} \tilde D_{3,2}(1,\la_{n,1,\eps},\la) \Phi_2(1, \la_{n,1,\eps}) \\ \label{Phi2la} & + \sum_{(n,2,\eps) \in V^N} (-1)^{\eps}\be_{n,2,\eps} \tilde D_{2,2}(1,\la_{n,2,\eps}, \la) \Phi_3(1, \la_{n,2,\eps}) - \sum_{n \in K} \ga_n \tilde D_{3,2}(1,\la_n, \la) \Phi_3(1,\la_n),
\end{align}
where $\la_n = \la_{n,1,0} = \la_{n,2,0}$ for $n \in K$. Using \eqref{defD}, \eqref{wron}, and the relations $\tilde \Phi_2(1,\la) = 0$, $\tilde \Phi^{\star}_2(1,\la) = 0$, we obtain
\begin{align} \label{smD22}
    \tilde D_{2,2}(1,\la_{n,2,\eps},\la) & = -\frac{\tilde \Phi^{\star [1]}(1,\la_{n,2,\eps}) \tilde \Phi_2'(1,\la)}{\la - \la_{n,2,\eps}}, \\ \nonumber
    \tilde D_{3,2}(1,\la_{n,1,\eps},\la) & = -\frac{\tilde \Phi_3^{\star[1]}(1,\la_{n,1,\eps}) \tilde \Phi_2'(1,\la) - \tilde \Phi_3(1,\la_{n,1,\eps}) \tilde \Phi_2^{[2]}(1,\la)}{\la - \la_{n,1,\eps}}
\end{align}

Note that $\tilde \Phi_2^{\star[1]}(1, \la_{n,2,1}) = 0$, $\tilde \Phi_3^{\star}(1,\la_{n,1,1}) = 0$. In addition, for $\la = \la_{n_0,1,1}$, we have $\tilde \Phi_2'(1,\la_{n_0,1,1}) = 0$, so
\begin{equation} \label{sm2}
\tilde D_{2,2}(1,\la_{n,2,\eps}, \la_{n_0,1,1}) = 0, \quad
\tilde D_{3,2}(1, \la_{n,1,1}, \la_{n_0,1,1}) = 0 \quad \text{if} \:\: n_0 \ne n.
\end{equation}
It can be shown that
$$
\be_{n,1,1} \tilde D_{3,2}(1,\la_{n,1,1}, \la_{n,1,1}) = 1.
$$

Hence, \eqref{Phi2la} implies
\begin{align*}
\Phi_2(1,\la_{n_0,1,1}) & = -\sum_{n = 1}^N \beta_{n,1,0} \frac{\tilde \Phi_3^{\star}(1,\la_{n,1,0}) \tilde \Phi_2^{[2]}(1,\la_{n_0,1,1})}{\la_{n_0,1,1} - \la_{n,1,0}} \Phi_2(1,\la_{n,1,0}) \\ & + \Phi_2(1,\la_{n_0,1,1}) - \sum_{n \in K} \ga_n \frac{\tilde \Phi_3^{\star}(1,\la_n) \tilde \Phi_2^{[2]}(1,\la_{n_0,1,1})}{\la_{n_0,1,1} - \la_n} \Phi_3(1,\la_n), \quad n_0 = \overline{1,N}.
\end{align*}

Simplifying this relation, we obtain
$$
\sum_{n = 1}^N \beta_{n,1,0} \frac{\tilde \Phi_3^{\star}(1,\la_{n,1,0}) \Phi_2(1,\la_{n,1,0})}{\la_{n_0,1,1} - \la_{n,1,1}} + \sum_{n \in K} \ga_n \frac{\tilde \Phi_3^{\star}(1,\la_n) \Phi_3(1,\la_n)}{\la_{n_0,1,1} - \la_n} = 0, \quad n_0 = \overline{1,N}.
$$

Thus, the following auxiliary function
\begin{equation} \label{defG1}
\mathcal G_1(\la) := \sum_{n = 1}^N \beta_{n,1,0} \frac{\tilde \Phi_3^{\star}(1,\la_{n,1,0}) \Phi_2(1,\la_{n_0,1,1})}{\la - \la_{n,1,1}} + \sum_{n \in K} \ga_n \frac{\tilde \Phi_3^{\star}(1,\la_n) \Phi_3(1,\la_n)}{\la - \la_n}
\end{equation}
has zeros $\{ \la_{n_0,1,1} \}_{n_0 = 1}^N$. At the same time, $\mathcal G_1(\la)$ has the poles $\{ \la_{n_0,1,0} \}_{n_0 = 1}^N$ (including $\la_n = \la_{n,1,0}$ for $n \in K$). Hence the following function is entire in $\la$:
$$
\mathcal G_1^{\diamond}(\la) := \mathcal G_1(\la) \prod_{n = 1}^N \frac{(\la - \la_{n,1,0})}{(\la - \la_{n,1,1})}.
$$

Obviously, $|\mathcal G_1(\la)| \to 0$ as $|\la| \to \infty$, and so $|\mathcal G_1^{\diamond}(\la)| \to 0$ as $|\la| \to \infty$. By Liouville's theorem, $\mathcal G_1^{\diamond}(\la) \equiv 0$, and so $\mathcal G_1(\la) \equiv 0$. Hence
\begin{align*}
\Res_{\la = \la_{n,1,1}} \mathcal G_1(\la) = \beta_{n,1,0} \tilde \Phi_3^{\star}(1,\la_{n,1,0}) \Phi_2(1,\la_{n,1,0}) = 0 \quad & \Rightarrow \quad \Phi_2(1,\la_{n,1,0}) = 0, \quad n = \overline{1,N} \setminus K, \\
\Res_{\la = \la_n} \mathcal G_1(\la) = \ga_n \tilde \Phi_3^{\star}(1,\la_n) \Phi_3(1,\la_n) = 0 \quad & \Rightarrow \quad \Phi_3(1,\la_n) = 0, \quad n \in K.
\end{align*}
Analogously, considering the derivative $\Phi_2'(1,\la_{n_0,1,1})$, we obtain the relations \eqref{step1} for $\nu = 1$.

\smallskip

\textsc{Step 2}. Let us prove that
\begin{equation} \label{step2}
\Phi_2(1,\la_{n,1,1}) = 0, \quad n \le N, \qquad
\Phi_3(1,\la_{n,2,0}) = 0, \quad n = \overline{1,N} \setminus K.
\end{equation}
Putting $\la = \la_{n_0,1,0}$ in \eqref{Phi2la}, we arrive at the relation
\begin{align*}
\Phi_2(1,\la_{n_0,1,0}) = & -\sum_{n = 1}^N \beta_{n,1,1} \frac{\tilde \Phi_3^{\star [1]}(1,\la_{n,1,1}) \tilde \Phi_2'(1,\la_{n_0,1,0})}{\la_{n_0,1,0}- \la_{n,1,1}} \Phi_2(1,\la_{n,1,1}) \\ & - \sum_{n = 1}^N \beta_{n,2,0} \frac{\tilde \Phi_2^{\star[1]}(1,\la_{n,2,0}) \tilde \Phi_2'(1,\la_{n_0,1,0})}{\la_{n_0,1,0} - \la_{n,2,0}} \Phi_3(1,\la_{n,2,0}) = 0.
\end{align*}

Hence, for $n_0 = \overline{1,N} \setminus K$, we have
\begin{equation} \label{sm1}
\sum_{n = 1}^N \beta_{n,1,1} \frac{\tilde \Phi_3^{\star[1]}(1,\la_{n,1,1}) \Phi_2(1,\la_{n,1,1})}{\la_{n_0,1,0} - \la_{n,1,1}} + \sum_{n = 1}^N \beta_{n,2,0} \frac{\tilde \Phi_2^{\star [1]}(1,\la_{n,2,0}) \Phi_3(1,\la_{n,2,0})}{\la_{n_0,1,0} - \la_{n,2,0}} = 0.
\end{equation}

Using \eqref{defPhiN}, \eqref{relPhiN}, and \eqref{step1}, we obtain
\begin{align} \nonumber
    \Phi_3(1,\la) = \tilde \Phi_3(1,\la) & + \sum_{n = 1}^N \beta_{n,1,1} \tilde D_{3,3}(1,\la_{n,1,1},\la) \Phi_2(1,\la_{n,1,1}) \\ \label{Phi3la} & + \sum_{(n,2,\eps) \in V^N} (-1)^{\eps} \beta_{n,2,\eps} \tilde D_{2,3}(1,\la_{n,2,\eps}, \la) \Phi_3(1,\la_{n,2,\eps}). 
\end{align}

Put $\la = \la_{n_0,2,1}$. Using the relations \eqref{wron}, \eqref{defD}, $\tilde \Phi_3(1,\la_{n_0,2,1}) = 0$, $\tilde \Phi_3^{\star}(1,\la_{n,1,1}) = 0$, $\tilde \Phi_2(1,\la) \equiv 0$, and $\tilde \Phi_2^{\star[1]}(1,\la_{n_0,2,1}) = 0$, we derive
\begin{align*}
    \tilde D_{3,3}(1,\la_{n,1,1}, \la_{n_0,1,0}) & = -\frac{\tilde \Phi_3^{\star}(1,\la_{n,1,1}) \tilde \Phi_3'(1,\la_{n_0,2,1})}{\la_{n_0,2,1} - \la_{n,1,1}}, \\
    \tilde D_{2,3}(1, \la_{n,2,0}, \la_{n_0,2,1}) & = -\frac{\tilde \Phi_2^{\star [1]}(1,\la_{n,2,0}) \tilde \Phi_3'(1,\la_{n_0,2,1})}{\la_{n_0,2,1} - \la_{n,2,0}}, \\
    \beta_{n,2,1} \tilde D_{2,3}(1, \la_{n,2,1}, \la_{n_0,2,1}) & = \begin{cases}
                            -1, \quad n = n_0, \\
                            0, \quad n \ne n_0.
                      \end{cases} 
\end{align*}

Substituting these relations into \eqref{Phi3la}, we get
\begin{align*}
    \Phi_3(1,\la_{n_0,2,1}) = & -\sum_{n = 1}^N \beta_{n,1,1} \frac{\tilde \Phi_3^{\star}(1,\la_{n,1,1}) \tilde \Phi_3'(1,\la_{n_0,2,1})}{\la_{n_0,2,1} - \la_{n,1,1}} \Phi_2(1,\la_{n,1,1}) \\ & - \sum_{n = 1}^N \beta_{n,2,0} \frac{\tilde \Phi_2^{\star [1]}(1,\la_{n,2,0}) \tilde \Phi_3'(1,\la_{n_0,2,1})}{\la_{n_0,2,1} - \la_{n,2,0}} \Phi_3(1, \la_{n,2,0}) + \Phi_3(1,\la_{n_0,2,1}).
\end{align*}

Consequently,
$$
\sum_{n = 1}^N \beta_{n,1,1} \frac{\tilde \Phi_3^{\star[1]}(1,\la_{n,1,1}) \Phi_2(1,\la_{n,1,1})}{\la_{n_0,2,1} - \la_{n,1,1}} + \sum_{n = 1}^N \beta_{n,2,0} \frac{\tilde \Phi_2^{\star [1]}(1,\la_{n,2,0}) \Phi_3(1,\la_{n,2,0})}{\la_{n_0,2,1} - \la_{n,2,0}} = 0, \quad n_0 = \overline{1,N}.
$$

Comparing the latter relation with \eqref{sm1}, we conclude that the following auxiliary function has zeros $\{ \la_{n_0,1,0} \}_{n_0 = \overline{1,N} \setminus K}$ and $\{ \la_{n_0,2,1}\}_{n_0 = 1}^N$:
$$
\mathcal G_2(\la) := \sum_{n = 1}^N \beta_{n,1,1} \frac{\tilde \Phi_3^{\star[1]}(1,\la_{n,1,1}) \Phi_2(1,\la_{n,1,1})}{\la - \la_{n,1,1}} + \sum_{n = 1}^N \beta_{n,2,0} \frac{\tilde \Phi_2^{\star [1]}(1,\la_{n,2,0}) \Phi_3(1,\la_{n,2,0})}{\la - \la_{n,2,0}}.
$$

At the same time, $\mathcal G_2(\la)$ has the poles $\{ \la_{n,1,1} \}_{n = 1}^N$ and $\{ \la_{n,2,0} \}_{n = \overline{1,N} \setminus K}$ (recall that $\beta_{n,2,0} = 0$ for $n \in K$). Analyzing the function $\mathcal G_2(\la)$ similarly to $\mathcal G_1(\la)$ defined by \eqref{defG1}, we show that $\mathcal G_2(\la) \equiv 0$. Hence
\begin{align*}
    \Res_{\la = \la_{n,1,1}} \mathcal G_2(\la) = \be_{n,1,1} \tilde \Phi_3^{\star[1]}(1,\la_{n,1,1}) \Phi_2(1, \la_{n,1,1}) = 0 \quad & \Rightarrow \quad \Phi_2(1,\la_{n,1,1}) = 0, \quad n \le N, \\
    \Res_{\la = \la_{n,2,0}} \mathcal G_2(\la) = \be_{n,2,0} \tilde \Phi_2^{\star[1]}(1,\la_{n,2,0}) \Phi_3(1,\la_{n,2,0}) = 0 \quad & \Rightarrow \quad \Phi_3(1,\la_{n,2,0}) = 0, \quad n = \overline{1,N} \setminus K.
\end{align*}

\textsc{Step 3}. Let us prove that 
$\Phi_2(1, \la) \equiv 0$, $\Phi_1(1,\la) \equiv 0$, and $\Phi_1'(1, \la) \equiv 0$.

Recall the assumption $\beta_{n,1,0} = \beta_{n,2,0} = 0$ for $n \in K$. Using \eqref{smD22}, we get $\tilde D_{2,2}(1, \la_{n,2,1}, \la) = 0$ for $\la \ne \la_{n,2,1}$. Consequently, using \eqref{Phi2la} and taking \eqref{step1}, \eqref{step2} into account, we obtain $\Phi_2(1,\la) \equiv 0$. Using \eqref{defPhiN}, \eqref{relPhiN}, the boundary conditions $\tilde \Phi_1^{(\nu)}(1,\la) = 0$, $\nu = 0, 1$, and the relations 
$$
\tilde D_{2,1}^{(\nu)}(1,\mu, \la) \equiv 0, \: \nu = 0, 1, \quad
\tilde D_{3,1}'(1, \mu, \la) \equiv 0,
$$
which can be easily checked, we get
\begin{align*}
    \Phi_1(1,\la) = & -\sum_{(n,1,\eps) \in V^N} (-1)^{\eps}\be_{n,1,\eps} \tilde D_{3,1}(1, \la_{n,1,\eps}, \la) \Phi_2(1, \la_{n,1,\eps}) - \sum_{n \in K} \ga_n \tilde D_{3,1}(1,\la_n,\la) \Phi_3(1,\la_n), \\
    \Phi_1'(1,\la) = & -\sum_{n = 1}^N \be_{n,1,0} \tilde D_{3,1}(1, \la_{n,1,0}, \la) \Phi_2'(1, \la_{n,1,0}) - \sum_{n \in K} \ga_n \tilde D_{3,1}(1, \la_n, \la) \Phi_3'(1, \la_n).
\end{align*}

Taking \eqref{step1} and \eqref{step2} into account, we conclude that $\Phi_1(1, \la) \equiv \Phi_1'(1,\la) \equiv 0$. 

\medskip

\textsc{Step 4}. Summarizing the results of steps~1--3, we conclude that $\Phi_k^N(x, \la)$, $k = 1, 2, 3$, fulfill the boundary conditions \eqref{bcPhiN}. Together with Lemma~\ref{lem:PhiNeq}, this implies that $\Phi_k^N(x, \la)$, $k = 1, 2, 3$,
are the Weyl solutions of equation~\eqref{eqvN}. It remains to show that $\mathfrak S^N$ are the corresponding spectral data. For this purpose, it is sufficient to prove the relations of Lemma~\ref{lem:WSD} for $\Phi_k^N(x, \la)$ and $\mathfrak S^N$. As before, we omit the upper index $N$.

Let us prove that $\Phi_2(1, \la_{n,1}^N) = 0$ for $n \in \mathbb N \setminus K$. For $n \le N$, this relation has been already proved at step~1. For $n > N$, we obtain $\Phi_2(1, \la_{n,1}^N) = \Phi_2(1, \la_{n,1,1}) = 0$ from \eqref{Phi2la} by using \eqref{step1}, \eqref{step2}, and \eqref{sm2}.

It follows from \eqref{defPhiN}, \eqref{relPhiN}, and the boundary conditions for $\Phi_2$ and $\Phi_3$ at $x = 0$ that
$$
\Phi_1'(0,\la)  = \tilde \Phi_1'(0, \la) - \sum_{(n,1,\eps) \in V^N} (-1)^{\eps} \be_{n,1,\eps} \tilde D_{3,1}(0, \la_{n,1,\eps}, \la).
$$
Using \eqref{defD} and \eqref{wron}, we show that 
$$
\tilde D_{3,1}(0, \la_{n,1,\eps}, \la) = \frac{1}{\la - \la_{n,1,\eps}}.
$$
Hence
\begin{gather*}
\Phi'_1(0, \la) = \tilde \Phi'(0,\la) - \sum_{n = 1}^N \left( \frac{\be_{n,1}}{\la - \la_{n,1}} - \frac{\tilde \be_{n,1}}{\la - \tilde \la_{n,1}}\right), \\
\Res_{\la = \la_{n,1}} \Phi'_1(0, \la) = -\be_{n,1}, \: n \le N, \qquad
\Res_{\la = \tilde \la_{n,1}} \Phi'_1(0, \la) = -\tilde \be_{n,1}, \: n > N.
\end{gather*}

The other relations of Lemma~\ref{lem:WSD} can be derived similarly, which concludes the proof.
\end{proof}

Lemma~\ref{lem:PhiNbc} implies the assertion of Lemma~\ref{lem:sdt} for the ``truncated'' spectral data $\mathfrak S^N$. For passing to the case of the general $\mathfrak S$, we use the approximation approach.

\begin{lem} \label{lem:limtau}
Let $\tau_1$, $\tau_0$, $\tau_1^N$, and $\tau_0^N$ be defined by the formulas \eqref{rec1}, \eqref{rec0}, \eqref{rec1N}, and \eqref{rec0N}, respectively. Then
$$
\lim_{N \to \infty} \| \tau_1^N - \tau_1 \|_{L_2(0,1)} = 0, \quad
\lim_{N \to \infty} \| \tau_0^N - \tau_0 \|_{W_2^{-1}(0,1)} = 0.
$$
\end{lem}

\begin{proof}
Recall that $\psi(x)$ and $\psi^N(x)$ are the solutions of the main equations $(I - \tilde R(x)) \psi(x) = \tilde \psi(x)$ and $(I - \tilde R^N(x)) \psi^N(x) = \tilde \psi(x)$, respectively, and the operators $\tilde R(x)$ and $\tilde R^N(x)$ fulfill the estimate $\| \tilde R(x) - \tilde R^N(x) \|_{m \to m} \le \eps_N$, where $\eps_N = o\left( N^{-1}\right)$, $N \to \infty$ (see Corollary~\ref{cor:RN}). Hence, for sufficiently large $N$, we have $\| \psi(x) - \psi^N(x) \|_m \le C \eps_N$. The functions $\vv_{n,k,\eps}$ and $\vv^N_{n,k,\eps}$, $n \le N$, are obtained from $\psi(x)$ and $\tilde \psi^N(x)$, respectively, by the corresponding formulas \eqref{defvvpsi} and \eqref{defvvpsiN}. Consequently, for $(n,k,\eps) \in V^N$, $\nu = 0, 1$, one can obtain the estimates
\begin{equation} \label{estvvN}
\left.
\begin{array}{c}
    |\vv^{(\nu)}_{n,k,\eps}(x) - (\vv^N_{n,k,\eps})^{(\nu)}(x)| \le C n^{\nu} w_{n,k}(x) \eps_N, \\
    |\vv^{(\nu)}_{n,k,0}(x) - \vv^{(\nu)}_{n,k,1}(x) - (\vv^N_{n,k,0})^{(\nu)}(x) + (\vv^N_{n,k,1})^{(\nu)}(x)| \le C n^{\nu} w_{n,k}(x) \xi_n \eps_N,
\end{array} \quad \right\}
\end{equation}
where the constant $C$ does not depend on $N$, $x$, $n$, $k$, and $\eps$.

The series $(\tau_1 - \tau_1^N)$ can be represented as the sum $-\frac{3}{2}(\mathscr T_1^N + \mathscr T_2^N)$, where
\begin{align*}
& \mathscr T_1^N := \sum_{(n,k,\eps) \in V^N} (-1)^{\eps} \bigl((\vv_{n,k,\eps}' - (\vv_{n,k,\eps}^N)') \tilde \eta_{n,k,\eps} + (\vv_{n,k,\eps} - \vv_{n,k,\eps}^N) \tilde \eta'_{n,k,\eps}\bigr), \\
& \mathscr T_2^N := \sum_{(n,k,\eps) \in V \setminus V^N} (-1)^{\eps} (\vv_{n,k,\eps}' \tilde \eta_{n,k,\eps} + \vv_{n,k,\eps} \tilde \eta_{n,k,\eps}').
\end{align*}

Using the estimates \eqref{esteta} and \eqref{estvvN}, we obtain
$$
|\mathscr T_1^N| \le C \eps_N \sum_{n = 1}^N (n \xi_n) \le C N \eps_N.
$$

Hence, $\mathscr T_1^N \to 0$ as $N \to \infty$ uniformly with respect to $x \in [0,1]$. Furthermore, Lemma~\ref{lem:series} implies that $\lim\limits_{N \to \infty} \| \mathscr T_2^N \|_{L_2(0,1)} = 0$. This concludes the proof for $(\tau_1 - \tau_1^N)$. The proof for $(\tau_0 - \tau_0^N)$ is analogous.
\end{proof}

\begin{lem} \label{lem:limsd}
Let $\tau_1$ and $\tau_1^N$, $N \ge 1$, be arbitrary functions of $L_2(0,1)$ such that $\tau_1^N \to \tau_1$ in $L_2(0,1)$ as $N \to \infty$, and let $\tau_0$ and $\tau_0^N$, $N \ge 1$, be arbitrary functions of $W_2^{-1}(0,1)$ such that $\tau_0^N \to \tau_0$ in $W_2^{-1}(0,1)$ as $N \to \infty$. Then, the following assertions are valid for fixed indices $n \in \mathbb N$ and $k \in \{ 1, 2 \}$:

\begin{enumerate}
\item Let $\la_{n,k}$ be an eigenvalue of multiplicity $m_{n,k}$ of the problem $\mathcal L_k$ with the coefficients $\mathcal T = (\tau_0, \tau_1)$.
Let $d_{\de} = \{ \la \in \mathbb C \colon |\la - \la_{n,k}| \le \de \}$ be a disk of sufficiently small radius $\de > 0$ which contains no other eigenvalues of $\mathcal L_k$ except for $\la_{n,k}$.
Then, for every sufficiently large $N$, the analogous problem $\mathcal L_k^N$ with the coefficients $\mathcal T^N = (\tau_0^N, \tau_1^N)$ has exactly $m_{n,k}$ eigenvalues (counting with multiplicities) in the disk $d_{\de}$.
\item Let $\la_{n,k}$ be a simple eigenvalue of $\mathcal L_k$. Then the corresponding problems $\mathcal L_k^N$ have simple eigenvalues $\la_{n,k}^N$ such that $\la_{n,k}^N \to \la_{n,k}$ as $N \to \infty$. Furthermore, $\be_{n,k}^N \to \be_{n,k}$ as $N \to \infty$.
\item Suppose that the problems $\mathcal L_1$ and $\mathcal L_2$ have a common simple eigenvalue $\la_{n,1} = \la_{n,2} = \la_n$, the problems $\mathcal L_1^N$ and $\mathcal L_2^N$ have a common simple eigenvalue $\la_{n,1}^N = \la_{n,2}^N = \la_n^N$ for each sufficiently large $N$, $\la_n^N \to \la_n$ as $N \to \infty$ and $\beta_{n,1} = \beta_{n,1}^N = 0$ for all sufficiently large $N$ (or $\beta_{n,2} = \beta_{n,2}^N = 0$ for all sufficiently large $N$). Then $\ga_n^N \to \ga_n$ as $N \to \infty$.
\end{enumerate}
\end{lem}

Lemma~\ref{lem:limsd} is proved by the well-known method for obtaining continuous dependence of the spectral data on the boundary value problem coefficients (see \cite{KZ96}). In recent years, this method has been actively developed for various classes of differential operators (see, e.g., \cite{Ao22-3ord, Ao22-4ord}). Therefore, here we outline the proof of Lemma~\ref{lem:limsd} briefly.

\begin{proof}[Proof of Lemma~\ref{lem:limsd}]
Recall that the eigenvalues $\{ \la_{n,k} \}$ coincide with the zeros of the characteristic functions $\Delta_{k,k}(\la)$, which are composed as some determinants of the functions $C_k^{[j]}(1,\la)$ (see formulas \eqref{defDelta1} and \eqref{defDelta2}).  It follows from the definition of the solutions $C_k(x, \la)$, $k = 1, 2, 3$, that the vectors of their quasi-derivatives solve the initial value problems \eqref{initC},
where $\sigma_0(x)$ is a fixed antiderivative of $\tau_0(x)$, so $\sigma_0$ and $\tau_1$ belong to $L_2(0,1)$. Obviously, the functions $C_k^{[j]}(1, \la)$, $k = 1, 2, 3$, $j = 0, 1, 2$, depend continuously on the coefficients $(\sigma_0, \tau_1)$ and analytically on $\la$. Consequently, so do the functions $\Delta_{j,k}(\la)$, $1 \le k < j \le 3$, defined by \eqref{defDelta1} and \eqref{defDelta2}. 

Suppose that $k \in \{ 1,2\}$ is fixed.
Let $\la_{n,k}$ be a zero of $\Delta_{k,k}(\la)$ of multiplicity $m_{n,k}$, and let the disk $d_\de$ satisfy the conditions of this lemma. Clearly, there exists a constant $c_0 > 0$ such that $|\Delta_{k,k}(\la)| \ge c_0$ on the boundary of $d_{\de}$. 
By the hypothesis of the lemma, $\tau_0^N \to \tau_0$ in $W_2^{-1}(0,1)$ as $N \to \infty$. Therefore, since $\sigma_0$ is a fixed antiderivative of $\tau_0$, then one can choose antiderivatives $\sigma_0^N$ of the functions $\tau_0^N$ so that $\sigma_0^N \to \sigma_0$ in $L_2(0,1)$ as $N \to \infty$. Consequently, it follows from the above arguments, that $\Delta_{k,k}^N(\la) \to \Delta_{k,k}(\la)$ as $N \to \infty$ uniformly with respect to $\la \in d_{\de}$. Hence
$$
|\Delta_{k,k}^N(\la) - \Delta_{k,k}(\la)| < c_0, \quad \la \in d_{\de},
$$
for all sufficiently large values of $N$. Rouche's Theorem implies that $\Delta_{k,k}^N(\la)$ has exactly $m_{n,k}$ zeros (counting with multiplicities) in $d_{\de}$. This proves the first assertion of the lemma.

Recall that the weight numbers $\be_{n,k}$ and $\ga_n$ are defined by \eqref{defbeta} and \eqref{defga1}-\eqref{defga2}, respectively, as the residues of some meromorphic functions composed of $C_k^{[j]}(1,\la)$. This implies the assertions 2 and 3 of the lemma.
\end{proof}

\begin{proof}[Proof of Lemma~\ref{lem:sdt}]
Consider the functions $\tau_1$ and $\tau_0$ which are constructed by the initially given data $\mathfrak S$ by the formulas \eqref{rec1} and \eqref{rec0}, respectively. Denote by $\{ \la_{n,k}^{\bullet} \}_{n \in \mathbb N}$ the eigenvalues of the corresponding problem $\mathcal L_k$, $k = 1, 2$, for equation \eqref{eqv}. We have to prove that the eigenvalues $\{ \la_{n,k}^{\bullet} \}$ coincide with the values $\{ \la_{n,k} \}$ of the data $\mathfrak S$. For this purpose, consider the data $\mathfrak S^N$ defined by \eqref{defSN} and the functions $\tau_1^N$ and $\tau_0^N$ constructed by the formulas \eqref{rec1N} and \eqref{rec0N}, respectively. By virtue of Lemma~\ref{lem:limtau}, we have $\tau_1^N \to \tau_1$ in $L_2(0,1)$ and $\tau_0^N \to \tau_0$ in $W_2^{-1}(0,1)$ as $N \to \infty$. Lemma~\ref{lem:PhiNbc} implies that $\mathfrak S^N$ are the spectral data of $\mathcal T^N = (\tau_0^N, \tau_1^N)$, in particular, $\{ \la_{n,k}^N \}_{n = 1}^{\infty}$ are the eigenvalues of the problem $\mathcal L_k^N$ for $k = 1, 2$. Due to the first assertion of Lemma~\ref{lem:limsd}, the eigenvalues of $\mathcal L_k^N$ converge to the eigenvalues of $\mathcal L_k$, that is, $\la_{n,k}^N \to \la_{n,k}^{\bullet}$ as $N \to \infty$, for each fixed $n \in \mathbb N$ and $k = 1, 2$. At the same time, it follows from \eqref{deflaN} that $\la_{n,k}^N \to \la_{n,k}$ as $N \to \infty$. Hence $\la_{n,k}^{\bullet} = \la_{n,k}$. Thus, the problems $\mathcal L_k$, $k = 1, 2$, have simple spectra, so the spectral data $\mathfrak S^{\bullet}$ of $\mathcal T$ have the structure similar to $\mathfrak S$:
$$
\mathfrak S^{\bullet} = \bigl( \{ \la_{n,k} \}_{n \in \mathbb N, \, k = 1, 2}, \{ \be_{n,k}^{\bullet} \}_{n \in \mathbb N, \, k = 1, 2}, \{ \ga_n^{\bullet} \}_{n \in K}\bigr).
$$

Successively applying the second and the third assertions of Lemma~\ref{lem:limsd} together with \eqref{defSN} and \eqref{deflaN}, we conclude that $\be_{n,k}^{\bullet} = \be_{n,k}$ for all $n \in \mathbb N$, $k = 1, 2$ and $\ga_n^{\bullet} = \ga_n$ for $n \in K$. This concludes the proof.
\end{proof}

Lemma~\ref{lem:sdt} finishes the proof of Theorem~\ref{thm:nsc}.

\begin{proof}[Proof of Theorem~\ref{thm:loc}]
Let $\tilde{\mathcal T}$ satisfy the hypothesis of Theorem~\ref{thm:loc}. Then, the conditions 1 and 2 of Theorem~\ref{thm:nsc} are valid for the spectral data $\tilde{\mathfrak S} = \{ \tilde \la_{n,k}, \tilde \be_{n,k} \}_{n \in \mathbb N, \, k = 1, 2}$ of $\tilde{\mathcal T}$ by the necessity. Therefore, the inequality \eqref{locsd} for sufficiently small $\eps > 0$ implies that the conditions 1 and 2 of Theorem~\ref{thm:nsc} hold for $\mathfrak S = \{ \la_{n,k}, \be_{n,k} \}_{n \in \mathbb N, \, k = 1, 2}$. Construct the main equation \eqref{main} by using the data $\mathfrak S$ and the model problem $\tilde{\mathcal T}$. Comparing \eqref{locsd} and \eqref{defxi}, we conclude that
\begin{equation} \label{xid}
\sqrt{\sum_{n = 1}^{\infty} (n\xi_n)^2} \le d(\mathfrak S, \tilde{\mathfrak S}) \le \eps.
\end{equation}

Using \eqref{estpsiR} and \eqref{xid}, we obtain the estimate
$$
\| \tilde R(x) \|_{m \to m} \le \sup_{n_0 \ge 1} \sum_{n = 1}^{\infty} \frac{C \xi_n}{|n - n_0| + 1} \le C \eps.
$$

Therefore, for sufficiently small $\eps > 0$, we have $\| \tilde R(x) \|_{m \to m} \le \frac{1}{2}$ for all $x \in [0,1]$. Hence, the operator $(I - \tilde R(x))$ has a bounded inverse. Thus, the data $\mathfrak S$ fulfill the conditions 1--3 of Theorem~\ref{thm:nsc}. This implies the existence of the coefficients $\mathcal T = (\tau_0, \tau_1) \in W$, for which $\mathfrak S$ are the spectral data ($K = \varnothing$).

The estimates \eqref{loctau} can be easily proved by using~\eqref{xid} and by following the proof of Lemma~\ref{lem:series}.
\end{proof}

\section{Self-adjoint case} \label{sec:sa}

The goal of this section is to prove Theorem~\ref{thm:suff} on the sufficient conditions of the inverse problem solvability. 
The central part in the proof is taken by Lemma~\ref{lem:solve} on the unique solvability of the main equation \eqref{main}.

We begin with some preliminaries.
Along with equation~\eqref{eqv} having the coefficients $\mathcal T = (\tau_0, \tau_1) \in W$, consider the analogous equation having the coefficients $\mathcal T^{\dagger} = (-\overline{\tau_0}, \overline{\tau_1})$: 
$$
\ell^{\dagger}(y) = y''' + (\overline{\tau_1(x)} y)' + \overline{\tau_1(x)} y' - \overline{\tau_0(x)} y = \la y, \quad x \in (0,1).
$$

The corresponding quasi-derivatives are induced by the associated matrix
$$
F^{\dagger}(x) = \begin{bmatrix}
            0 & 1 & 0 \\
            \overline{\sigma_0} - \overline{\tau_1} & 0 & 1 \\
            0 & -(\overline{\sigma_0} + \overline{\tau_1}) & 0
        \end{bmatrix}.
$$

Obviously, $F^{\dagger}(x) = \overline{F^{\star}(x)}$, where $F^{\star}(x)$ was defined in \eqref{defF*}. Therefore, there is the one-to-one correspondence $y(x, \la) = \overline{z(x, -\overline{\la})}$ between solutions $y(x, \la)$ and $z(x, \la)$ of the equations $\ell^{\dagger}(y) = \la y$ and $\ell^{\star}(z) = \la z$, respectively. Hence $M_{j,k}^{\dagger}(\la) = \overline{M_{j,k}^{\star}(-\overline{\la})}$, $1 \le k < j \le 3$. Using Corollary~\ref{cor:M*}, we obtain the following relations between the spectral data of $\mathcal T^{\dagger}$ and $\mathcal T$:
\begin{equation} \label{sd+}
\la_{n,1}^{\dagger} = -\overline{\la_{n,2}}, \quad \la_{n,2}^{\dagger} = -\overline{\la_{n,1}}, \quad \be_{n,1}^{\dagger} = -\overline{\be_{n,2}}, \quad \be_{n,2}^{\dagger} = -\overline{\be_{n,1}}, \quad \ga_n^{\dagger} = \overline{\ga_n}.
\end{equation}

Now suppose that $\mathrm i \tau_0(x)$ and $\tau_1(x)$ are real-valued functions. Then $\mathcal T^{\dagger} = \mathcal T$ and so \eqref{sd+} implies
$$
\la_{n,1} = -\overline{\la_{n,2}}, \quad \be_{n,1} = -\overline{\be_{n,2}}.
$$

As before, we assume that the eigenvalues $\{ \la_{n,k} \}_{n = 1}^{\infty}$ for each $k \in \{ 1, 2\}$ are simple. However, the eigenvalue order changes. Since we suppose that $\la_{n,1} = -\overline{\la_{n,2}}$, then $\la_{n,1} = \la_{p,2}$ not necessarily implies $n = p$. Therefore, we need some changes in the definition of $\ga_n$.

Put $\la_n := \la_{n,1}$, $\be_n := \be_{n,1}$, and define the index set
$$
K^+ := \{ n \in \mathbb N \colon \exists p = p(n) \: \textit{s.t.} \: \la_n = -\overline{\la_p} \}.
$$

If $\be_n = 0$, then define $\ga_n$ by formula \eqref{defga1}. If $n \in K^+$ and $\be_n \ne 0$, then $\la_{p(n),2} = \la_{n,1}$ and $\be_{p(n),2} = -\overline{\be_{p(n)}} = 0$, so we define $\ga_n$ by \eqref{defga2}. It follows from \eqref{sd+} that $\ga_n = \overline{\ga_{p(n)}}$ for $n \in K^+$. 
If $n = p(n)$, then $\la_n$ is purely imaginary, $\be_n = 0$, $\ga_n$ is real, and, moreover, it can be shown that $\ga_n > 0$. Our main goal is to prove Theorem~\ref{thm:suff}, so below we confine ourselves to the case $n = p(n)$ for all $n \in K^+$. Obviously, in this case, $K = K^+$. 

Proceed to the proof of Theorem~\ref{thm:suff}. Let $\mathfrak S^+ = (\{ \la_n \}_{n = 1}^{\infty}, \{ \be_n \}_{n \in \mathbb N \setminus K}, \{ \ga_n \}_{n \in K})$ be arbitrary numbers satisfying the hypothesis of Theorem~\ref{thm:suff}. Denote
\begin{equation} \label{defsd}
\left. \begin{array}{c}
\la_{n,1} := \la_n, \quad \la_{n,2} = -\overline{\la_n}, \quad n \in \mathbb N, \\
\be_{n,1} := \be_n, \quad \be_{n,2} := -\overline{\be_n}, \quad n \in \mathbb N \setminus K, \\
\be_{n,1} := 0, \quad \be_{n,2} := 0, \quad N \in K.
\end{array} \quad \right\}
\end{equation}

Consider the data $\mathfrak S = (\{ \la_{n,k} \}_{n \in \mathbb N, \, k = 1, 2}, \{ \be_{n,k} \}_{n \in \mathbb N, \, k = 1, 2}, \{ \ga_n \}_{n \in K})$. Choose a model problem $\tilde{\mathcal T} = (\tilde \tau_0, \tilde \tau_1)$ satisfying the conditions of Section~\ref{sec:equ} and the additional condition that $\mathrm i \tilde \tau_0(x)$ and $\tilde \tau_1(x)$ are real-valued. Following the steps of Section~\ref{sec:equ}, construct the operator $\tilde R(x) \colon m \to m$ and the elements $\tilde \psi(x) \in m$ and consider the main equation~\eqref{main}. For the proof of Theorem~\ref{thm:suff}, the following lemma on the unique solvability of the main equation is crucial.

\begin{lem} \label{lem:solve}
Under the conditions of Theorem~\ref{thm:suff}, the main equation \eqref{main} is uniquely solvable for each fixed $x \in [0,1]$.
\end{lem}

\begin{proof}
Let $x \in [0,1]$ be fixed.
By virtue of Proposition~\ref{prop:maineq}, the operator $\tilde R(x)$ has the approximation property. Therefore, in view of Fredholm's Theorem, it is sufficient to prove that the homogeneous equation
\begin{equation} \label{homo}
(I - \tilde R(x)) \zeta(x) = 0,
\end{equation}
has the unique solution $\zeta(x) = 0$ in $m$.

Let $\zeta(x) = [\zeta_v(x)]_{v \in V} \in m$ be a solution of \eqref{homo}. Then
$$
\zeta_{v_0}(x) = \sum_{v \in V} \tilde R_{v_0,v}(x) \zeta_v(x), \quad |\zeta_{v_0}(x)| \le C, \quad v_0 \in V.
$$

Analogously to \eqref{defvvpsi}, define
$$
\begin{bmatrix}
    z_{n,k,0}(x) \\ z_{n,k,1}(x)
\end{bmatrix} := w_{n,k}(x)
\begin{bmatrix}
    \xi_n & 1 \\ 0 & 1
\end{bmatrix}
\begin{bmatrix}
    \zeta_{n,k,0}(x) \\ \zeta_{n,k,1}(x)
\end{bmatrix}.
$$

Thus, $\zeta_{n,k,\eps}(x)$ is the analog of $\psi_{n,k,\eps}(x)$ and $z_{n,k,\eps}(x)$ is the analog of $\vv_{n,k,\eps}(x)$ for the homogeneous equation \eqref{homo}.
Consequently, we have
\begin{gather} \label{sysz1}
z_{n_0,k_0,\eps_0}(x) = \sum_{(n,k,\eps) \in V} (-1)^{\eps} z_{n,k,\eps}(x) \tilde G_{(n,k,\eps), (n_0,k_0,\eps_0)}(x),
\quad (n_0,k_0,\eps_0) \in V, \\ \label{estz}
|z_{n,k,\eps}(x)| \le C w_{n,k}(x), \quad |z_{n,k,0}(x) - z_{n,k,1}(x)| \le C w_{n,k}(x) \xi_n, \quad (n,k,\eps) \in V.
\end{gather}

Using \eqref{defG} and taking into account that $\be_{n,k,0} = 0$ for $k \in K$, we can consider the following system instead of \eqref{sysz1}:
\begin{align} \nonumber
    z_{n_0, k_0, \eps_0}(x) = & \sum_{(n,k,\eps) \in V^-} (-1)^{\eps + k} \be_{n,k,\eps} \tilde D_{4 - k,k_0+1}(x, \la_{n,k,\eps}, \la_{n_0,k_0,\eps_0}) z_{n,k,\eps}(x) \\ \label{sysz2} & - \sum_{n \in K} \ga_n \tilde D_{3,k_0+1}(x, \la_n, \la_{n_0,k_0,\eps_0}) z_n(x), \qquad (n_0,k_0,\eps_0) \in V^-, \\ \nonumber
    z_{n_0}(x) = & \sum_{(n,k,\eps) \in V^-} (-1)^{\eps+k} \be_{n,k,\eps} \tilde D_{4-k,3}(x, \la_{n,k,\eps}, \la_{n_0}) z_{n,k,\eps}(x) \\ \label{syszn} & - \sum_{n \in K} \ga_n \tilde D_{3,3}(x, \la_n, \la_{n_0}) z_n(x), \qquad n_0 \in K,
\end{align}
where
$$
V^- := V \setminus \{ (n,k,0) \colon n \in K, \, k= 1, 2 \}, \quad z_n(x) := z_{n,2,0}(x), \quad n \in K.
$$

Obviously, if the solution $(\{ z_{n,k,\eps}(x) \}_{(n,k,\eps) \in V^-}, \{ z_n(x) \}_{n \in K})$ of the system \eqref{sysz2}--\eqref{syszn} is zero, then the solution $\{ z_{n,k,\eps}(x) \}_{(n,k,\eps) \in V}$ of \eqref{sysz1} is also zero, and so does $\zeta(x)$. Let us prove this. 

Consider the functions
\begin{equation} \label{defZ}
Z_{k_0}(x, \la) := \sum_{(n,k,\eps) \in V^-} (-1)^{\eps + k} \be_{n,k,\eps} \tilde D_{4 - k,k_0}(x, \la_{n,k,\eps}, \la) z_{n,k,\eps}(x) - \sum_{n \in K} \ga_n \tilde D_{3,k_0}(x, \la_n, \la) z_n(x)
\end{equation}
for $k_0 = 1, 2, 3$. In view of \eqref{sysz2} and \eqref{syszn}, we have
\begin{equation} \label{relZ}
Z_{k+1}(x, \la_{n,k,\eps}) = z_{n,k,\eps}(x), \quad (n,k,\eps) \in V^-, \qquad
Z_3(x, \la_n) = z_n(x), \quad n \in K.
\end{equation}

Furthermore, the function $Z_3(x, \la)$ is entire in $\la$, and the functions $Z_1(x, \la)$ and $Z_2(x, \la)$ are meromorphic with the simple poles $\{ \la_{n,1,0} \}_{n \in N}$ and $\{ \la_{n,2,0} \}_{n \in \mathbb N \setminus K}$, respectively. Calculations show that
\begin{align} \label{ResZ1}
\Res_{\la = \la_{n,k,0}} Z_k(x, \la) & = -\be_{n,k,0} z_{n,k,0}(x), \quad (n,k,0) \in V^-, \\ \label{ResZ2} \Res_{\la = \la_n} Z_1(x, \la) & = -\ga_n z_n(x), \quad n \in K.
\end{align}

Fix $\de > 0$.
Using Lemma~\ref{lem:estPhi}, \eqref{asymptla}, \eqref{defD}, \eqref{esteta}, \eqref{estz}, and \eqref{defZ}, we obtain the estimate
\begin{gather} \label{estZ}
|Z_{k_0}(x, \rho^3)| \le \sum_{n = 1}^{\infty} \sum_{k = 1, 2} \frac{C \xi_n |\rho|^{-(k_0-1)}| |\exp(\rho \om_{k_0} x)|}{|\rho - \rho_{n,k}^0| + 1}, \quad \rho \in \overline{\Gamma_{s,\de}}, \quad k_0 = 1, 2, 3, \\ \label{defcont}
\Gamma_{s,\de} := \{ \rho \in \Gamma_s \colon |\rho| \ge \de, \, |\rho - \rho_{n,k}^0| \ge \de, \, n \in \mathbb N, \, k = 1, 2 \},
\end{gather}
where $\{ \om_k \}_{k = 1}^3$ are the roots of the equation $\om^3 = 1$ numbered according to \eqref{order}, $\rho_{n,k}^0 = \sqrt[3]{\la_{n,k}^0} \in \overline{\Gamma_s}$, $\{ \la_{n,k}^0 \}$ are the eigenvalues of the problem $\mathcal L_k^0$ with the zero coefficients $\mathcal T = (0,0)$.

Consider the functions
\begin{equation} \label{defB}
\mathcal B_1(x, \la) := -Z_1(x, \la) \overline{Z_3(x, -\overline{\la})}, \quad
\mathcal B_2(x, \la) := -Z_2(x, \la) \overline{Z_2(x, -\overline{\la})}.
\end{equation}

In view of the analytic properties of $Z_k(x, \la)$, $k = 1, 2, 3$, the functions $\mathcal B_1(x, \la)$ and $\mathcal B_2(x, \la)$ are meromorphic in $\la$ with the simple poles $\{ \la_{n,1,0} \}_{n \in \mathbb N}$ and $\{ \la_{n,k,0} \}_{n \in \mathbb N \setminus K, \, k = 1, 2}$, respectively. Using \eqref{relZ}, \eqref{ResZ1}, and \eqref{ResZ2}, we find the residues:
\begin{align*}
    \Res_{\la = \la_{n,1,0}} \mathcal B_1(x, \la) & = \be_{n,1,0} z_{n,1,0}(x) \overline{z_{n,2,0}(x)} =: r_n(x), \quad n \in \mathbb N \setminus K, \\
    \Res_{\la = \la_n} \mathcal B_1(x, \la) & = \ga_n z_n(x) \overline{z_n(x)}, \quad n \in K, \\
    \Res_{\la = \la_{n,1,0}} \mathcal B_2(x, \la) & = -\overline{\be_{n,2,0} z_{n,2,0}(x)} z_{n,1,0}(x) = r_n(x), \quad n \in \mathbb N \setminus K, \\
    \Res_{\la = \la_{n,2,0}} \mathcal B_2(x, \la) & = \be_{n,2,0} z_{n,2,0}(x) \overline{z_{n,1,0}(x)} = -\overline{r_n(x)}, \quad n \in \mathbb N \setminus K.
\end{align*}

The estimate \eqref{estZ} together with \eqref{defB} imply
\begin{equation} \label{estB}
|\mathcal B_j(x, \rho^3)| \le \frac{C}{|\rho|^2} \left( \sum_{n = 1}^{\infty} \sum_{k = 1, 2} \frac{\xi_n}{|\rho - \rho_{n,k}^0| + 1}\right)^2 \le \frac{C}{|\rho|^4}, \quad \rho \in \overline{\Gamma_{s,\de}}, \quad j = 1, 2.
\end{equation}

Consider the contours $\mathcal C_R := \{ \la \in \mathbb C \colon |\la| = R \}$ of sufficiently large radii $R$ such that $\la = \rho^3 \in \mathcal C_R$ imply $\rho \in \overline{\Gamma_{s,\de}}$ for some $s \in \{ 1, 2\}$ and the fixed $\de > 0$. Then
$$
\lim_{R \to \infty}\frac{1}{2 \pi \mathrm{i}} \oint_{\mathcal C_R} \mathcal B_1(x, \la) \, d\la = 0.
$$

Applying the Residue Theorem, we get
\begin{equation} \label{sumga}
\sum_{n \in \mathbb N \setminus K} r_n(x) + \sum_{n \in K} \ga_n |z_n(x)|^2 = 0.
\end{equation}

Since $\ga_n > 0$, then 
\begin{equation} \label{ineqr}
\sum_{n \in \mathbb N \setminus K} r_n(x) \le 0.
\end{equation}

Denote by $\mathcal C_R^+$ the arc $\{ \la \in \mathcal C_R \colon \mbox{Re} \, \la \ge 0 \}$ and consider the contour $\mathcal C_R^0 = \mathcal C_R^+ \cup [-\mathrm{i} R, \mathrm{i} R]$ with the counter-clockwise circuit (see Figure~\ref{img:contour}). According to the estimate \eqref{estB}, we have
$$
\lim_{R \to \infty} \frac{1}{2 \pi \mathrm{i}} \int_{\mathcal C_R^+} \mathcal B_2(x, \la) \, d\la = 0.
$$

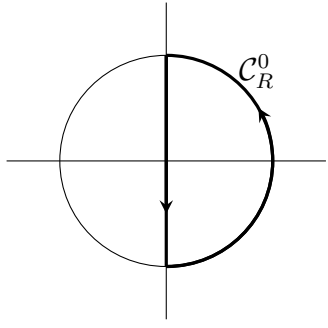
\begin{figure}[h!]
\centering
\begin{tikzpicture}[scale = 0.7]
\draw (-3,0) edge (3,0);
\draw (0,-3) edge (0,3);
\draw (0,0) circle (2);
\draw[very thick] (0, -2) arc(-90:90:2); 
\draw[very thick] (0, -2) edge (0, 2);
\draw[-stealth, very thick] (0, -2) arc(-90:30:2); 
\draw[-stealth, very thick] (0, 2) -- (0, -1);
\draw (1.7, 1.7) node{$\mathcal C_R^0$};
\end{tikzpicture}
\caption{Contour $\mathcal C_R^0$}
\label{img:contour}
\end{figure}

Hence
\begin{equation} \label{ineqG}
\lim_{R \to \infty} \frac{1}{2 \pi \mathrm{i}} \oint_{\mathcal C_R^0} \mathcal B_2(x, \la) \, d \la = -\frac{1}{2\pi \mathrm{i}} \int_{-\mathrm{i} \infty}^{\mathrm{i} \infty} \mathcal B_2(x, \la) \, d\la = \frac{1}{2\pi} \int_{-\infty}^{\infty} |\Gamma_2(x, \mathrm{i}\tau)|^2 \, d\tau \ge 0.
\end{equation}

The Residue Theorem implies
$$
\lim_{R \to \infty} \frac{1}{2 \pi \mathrm{i}} \oint_{\mathcal C_R^0} \mathcal B_2(x, \la) \, d \la = \sum_{n \in \mathbb N \setminus K} r_n(x).
$$

Taking the inequalities \eqref{ineqr} and \eqref{ineqG} into account, we conclude that $\Gamma_2(x, \la) \equiv 0$ for $\mathrm{i} \la \in \mathbb R$. By the analytic continuation principle, we have $\Gamma_2(x, \la) \equiv 0$, $\la \in \mathbb C$. The relations \eqref{ResZ1} and $\be_{n,2,0} \ne 0$ imply $z_{n,2,0}(x) = 0$ for $n \in \mathbb N \setminus K$. It follows from \eqref{relZ} that $z_{n,1,\eps}(x) = Z_2(x, \la_{n,1,\eps}) = 0$ for $(n,1,\eps) \in V^-$.
Then, the relation \eqref{sumga} implies
$$
\sum_{n \in K} \ga_n |z_n(x)|^2 = 0.
$$

Hence $z_n(x) = 0$ for $n \in K$. In view of \eqref{relZ}, the entire function $\Gamma_3(x, \la)$ has zeros $\{ \la_{n,2,0} \}_{n \in \mathbb N}$. Consider the infinite product
$$
P(\la) := \prod_{n = 1}^{\infty} \left( 1 - \frac{\la}{\la_{n,2,0}}\right).
$$

Using the asymptotics \eqref{asymptla}, one can show that
\begin{equation} \label{estP}
|P(\rho^3)| \ge C |\rho|^{-2} |\exp(\rho \om_3)|, \quad \rho \in \overline{\Gamma_{s,\de}}.
\end{equation}

Clearly, the function $\frac{Z_3(x, \la)}{P(\la)}$ is entire. The estimates \eqref{estZ} for $k_0 = 3$ and \eqref{estP} imply
$$
\left| \frac{Z_3(x, \rho^3)}{P(\rho^3)} \right| \le \frac{C}{|\rho|}, \quad \rho \in \overline{\Gamma_{s,\de}}.
$$

By Liouville's Theorem, we conclude that $Z_3(x, \la) \equiv 0$ and so $z_{n,2,1}(x) = Z_3(x, \la_{n,2,1}) = 0$, $n \in \mathbb N$.
Thus, $z_{n,k,\eps}(x) = 0$ for all $(n,k,\eps) \in V$. This yields the claim of the lemma.
\end{proof}

\begin{proof}[Proof of Theorem~\ref{thm:suff}]
It follows from the hypothesis of Theorem~\ref{thm:suff} and Lemma~\ref{lem:solve} that the data $\mathfrak S$ fulfill the conditions of Theorem~\ref{thm:nsc}. By virtue of Theorem~\ref{thm:nsc}, there exist coefficients $\mathcal T = (\tau_0, \tau_1) \in W$ with the spectral data $\mathfrak S$. It remains to show that $\mathcal T \in W^+$, that is, $\mathrm{i} \tau_0(x), \, \tau_1(x) \in \mathbb R$. For this purpose, consider the coefficients $\mathcal T^{\dagger} = (-\overline{\tau_0}, \overline{\tau_1})$ and the corresponding spectral data $\mathfrak S^{\dagger}$. Recall that the spectral data $\mathfrak S$ and $\mathfrak S^{\dagger}$ are related by \eqref{sd+}. Since \eqref{defsd} holds and $\ga_n \in \mathbb R$, $n \in K$, we obtain $\mathfrak S = \mathfrak S^{\dagger}$. By virtue of the uniqueness theorem (Theorem~\ref{thm:uniq}), this implies $\mathcal T = \mathcal T^{\dagger}$, which concludes the proof.
\end{proof}

\section{Conclusion} \label{sec:concl}

In this paper, the necessary and sufficient conditions of the inverse problem solvability are obtained for the third-order differential equation~\eqref{eqv}. We treat equation \eqref{eqv} by using the regularization approach of Mirzoev and Shkalikov \cite{MS16, MS19}. The proof of the main result (Theorem~\ref{thm:nsc}) is based on the constructive method of \cite{Yur02, Bond22-alg}. This method reduces the inverse problem to the linear main equation \eqref{main} in the Banach space of bounded infinite sequences $m$. Our NSC include asymptotic and structural properties of the spectral data and the requirement of the main equation unique solvability. In addition, we have studied the two special cases, in which the latter requirement can be achieved. The first one is the case of a small perturbation of the spectral data. The investigation of this case implies Theorem~\ref{thm:loc} on the local solvability and stability of the inverse spectral problem. The second case is the self-adjoint one. For this case, we prove Theorem~\ref{thm:suff}, which provides very simple sufficient conditions on the spectral data.

The results of this paper can be extended to various classes of higher-order differential operators generated by differential expression of form
\begin{align} \nonumber
\ell_s(y) := & y^{(s)} + \sum_{k = 0}^{\lfloor s/2\rfloor - 1} (\tau_{2k}(x) y^{(k)})^{(k)} \\ \label{defl} + & \sum_{k = 0}^{\lfloor (s-1)/2\rfloor - 1}  \bigl((\tau_{2k+1}(x) y^{(k)})^{(k+1)} + (\tau_{2k+1}(x) y^{(k+1)})^{(k)}\bigr),
\end{align}
where $s \ge 3$, the coefficients $\{ \tau_{\nu} \}_{\nu = 0}^{s-2}$ can be integrable or distributional, and the notation $\lfloor a \rfloor$ means rounding a real number $a$ down. 
However, such extension will be non-trivial and will cause additional difficulties. Let us discuss some of them.

If the coefficients $\{ \tau_{\nu} \}_{\nu = 0}^{s-2}$ are sufficiently smooth, then the spectral data asymptotics contain a large number of constant coefficients (see, e.g., \cite{Yur02}):
$$
\sqrt[s]{\la_{n,k}} = \sum_{j = -1}^p c_{j,k} n^{-j} +  n^{-(p+1)} \varkappa_{n,k}, \quad \be_{n,k} = n^{r_k} \left(\sum_{j = 0}^{p-1} b_{j,k} n^{-j} + n^{-(p+1)} \varkappa_{n,k}^0 \right).
$$

The constants $c_{j,k}$ and $b_{j,k}$ are related to the coefficients $\{ \tau_{\nu} \}_{\nu = 0}^{s-2}$, and it is a technical algebraic problem to obtain these relations. Additionally, in order to get NSC, one has to invent an algorithm for determining whether given numbers $c_{j,k}$ and $b_{j,k}$ can be the coefficients in the spectral data asymptotics for any boundary value problems or not. Consequently, for technical reasons, it is more convenient to consider the case of distribution coefficients $\{ \tau_{\nu} \}_{\nu = 0}^{s-2}$.

It is important to note that the Mirzoev-Shkalikov regularization has been obtained for the case $\tau_{\nu} = \sigma_{\nu}^{(i_{\nu})}$, $i_{2k+j} = l - k - j$, $l = \lfloor s/2 \rfloor$, $k \ge 1$, $j = 0, 1$,
$\sigma_{\nu} \in L_2(0,1)$ if $s = 2l$ and $\sigma_{\nu} \in L_1(0,1)$ if $s = 2l+1$ (see \cite{MS16, MS19}). For differential expression coefficients of higher singularity orders, there are no regularization results. However, in the Mirzoev-Shkalikov case for $s \ge 3$ and in the cases of lower singularity orders (e.g., $s = 2l$ and either $\sigma_{\nu} \in L_2(0,1)$ or $\sigma_{\nu} \in W_2^{-1}(0,1)$ for all $\nu = \overline{0,s-2}$), a step-by-step process of recovering the coefficients $\sigma_{s-2}$, $\sigma_{s-3}$, \dots, $\sigma_1$, $\sigma_0$ is needed (see \cite{Bond22-alg}). This makes it difficult to obtain NSC on the spectral data, since additional requirements should be imposed at each step.

The proof of the main equation solvability in the self-adjoint case (Theorem~\ref{thm:suff}) can be generalized to arbitrary odd orders. For even orders, the proof technique will be different. 
In \cite{Yur95-mn}, the unique solvability of the main equation has been proved for differential operators with regular coefficients of arbitrary even order on the half-line. For the finite interval, there are no such results. For odd orders, to the best of the author's knowledge, the proof of the main equation solvability in this paper is fundamentally new and has no analogs in previous studies.

Certainly, one can replace the boundary conditions \eqref{bc1} and \eqref{bc2} by separated boundary conditions of higher orders. But then the question of recovering the boundary condition coefficients arises, which is non-trivial for differential operators with distribution coefficients. As an example, consider the Sturm-Liouville equation \eqref{StL} with potential $q \in W_2^{-1}(0,1)$ and with the Robin boundary conditions
$$
y^{[1]}(0) - h y(0) = 0, \quad y^{[1]}(1) + H y(1) = 0,
$$
where $y^{[1]} := y' - \sigma y$, $q = \sigma'$. Then the both coefficients $h$ and $H$ cannot be uniquely recovered from the spectral data (see \cite{HM03}), because the spectral data are invariant with respect to the shift $\sigma := \sigma + c$, $h := h - c$, $H := H + c$. For higher orders, the situation is richer. The uniqueness issues have been studied in \cite{Bond22-half}, while the reconstruction of the boundary condition coefficients requires a separate investigation.

\medskip

{\bf Funding.} This work was supported by Grant 21-71-10001 of the Russian Science Foundation, https://rscf.ru/en/project/21-71-10001/.

\noindent Natalia Pavlovna Bondarenko \\
Department of Mechanics and Mathematics, Saratov State University, \\
Astrakhanskaya 83, Saratov 410012, Russia, \\
e-mail: {\it bondarenkonp@info.sgu.ru}


\medskip

\begin{thebibliography}{99}


\bibitem{BP19}
Braeutigam, I.N.; Polyakov, D.M. On the asymptotics of eigenvalues of a third-order differential operator, St. Petersburg Math. J. 31 (2020), no.~4, 585--606.

\bibitem{Kor19}
Korotyaev, E.L. Resonances of third order differential operators, J. Math. Anal. Appl. 478 (2019), no.~1, 82--107.

\bibitem{Ug19}
U\v{g}urlu, E. Regular third-order boundary value problems, Appl. Math. Comput. 343 (2019), 247--257.

\bibitem{Ug20}
U\v{g}urlu, E. Some singular third-order boundary value problems, Math. Meth. Appl. Sci. 43 (2020), no.~5, 2202--2215.

\bibitem{BK21}
Badanin, A.; Korotyaev, E.L. Third-order operators with three-point conditions associated with Boussinesq's equation, Appl. Anal. 100 (2021), no.~3, 527--560.

\bibitem{Ao22-3ord}
Zhang, H.-Y.; Ao, J.-J.; Mu, D. Eigenvalues of discontinuous third-order boundary value problems with eigenparameter-dependent boundary conditions, J. Math. Anal. Appl. 506 (2022), no.~2, 125680.

\bibitem{ZLW23}
Zhang, M.; Li, K.; Wang, Y. Regular approximation of singular third-order differential operators, J. Math. Anal. Appl. 521 (2023), no.~1, 126940.


\bibitem{Greg87}
Gregu\v{s}, M. Third Order Linear Differential Equations, Springer, Dordrecht (1987).

\bibitem{BP96}
Bernis, F.; Peletier, L.A. Two problems from draining flows involving third-order ordinary differential equations, SIAM J. Math. Anal. 27 (1996), no. 2, 515--527.

\bibitem{TS90}
Tuck, E.O.; Schwartz, L.W. A numerical and asymptotic study of some third-order ordinary differential equations relevant to draining and coating flows, SIAM Rev. 32 (1990), no. 3, 453--469.

\bibitem{McK81}
McKean, H. Boussinesq's equation on the circle, Comm. Pure Appl. Math. 34 (1981), no.~5, 599--691.


\bibitem{Mar77}
Marchenko, V.A. Sturm-Liouville Operators and Their Applications, Birkhauser (1986).

\bibitem{Lev84}
Levitan, B.M. Inverse Sturm-Liouville Problems, VNU Sci. Press, Utrecht (1987).

\bibitem{FY01}
Freiling, G.; Yurko, V. Inverse Sturm-Liouville Problems and Their Applications, Huntington, NY: Nova Science Publishers (2001).

\bibitem{Krav20}
Kravchenko, V.V. Direct and Inverse Sturm-Liouville Problems, Birkh\"auser, Cham (2020).


\bibitem{GL51}
Gel'fand, I.M.; Levitan, B.M. On the determination of a differential equation from its spectral function, Izv. Akad. Nauk SSSR, Ser. Mat. 15 (1951), 309-360 [in Russian].


\bibitem{Yur92}
Yurko, V.A. Recovery of nonselfadjoint differential operators on the half-line from the Weyl matrix, Math. USSR-Sb. 72 (1992), no.~2, 413--438.

\bibitem{Yur95-mn}
Yurko, V.A. On determination of self-adjoint differential operators on a semiaxis, Math. Notes 57 (1995), no.~3, 310--318.

\bibitem{Yur00}
Yurko, V. Inverse problems of spectral analysis for differential operators and their applications, J. Math. Sci. 98 (2000), no.~3, 319--426.

\bibitem{Yur02}
Yurko, V.A. Method of Spectral Mappings in the Inverse Problem Theory, Inverse and Ill-Posed Problems Series, Utrecht, VNU Science (2002).

\bibitem{Leib66}
Leibenson, Z.L. The inverse problem of spectral analysis for higher-order ordinary differential operators, Trudy Moskov. Mat. Obshch. 15 (1966), 70--144; English transl. in Trans. Moscow Math. Soc. 15 (1966).

\bibitem{Leib71}
Leibenson, Z.L. Spectral expansions of transformations of systems of boundary value problems, Trudy Moskov. Mat. Obshch. 25 (1971), 15--58; English transl. in Trans. Moscow Math. Soc. 25 (1971).

\bibitem{Beals85}
Beals, R. The inverse problem for ordinary differential operators on the line, American J. Math. 107 (1985), no.~2, 281--366.


\bibitem{Bond21}
Bondarenko, N.P. Inverse spectral problems for arbitrary-order differential operators with distribution coefficients, Mathematics 9 (2021), no. 22, Article ID 2989.

\bibitem{Bond22-half}
Bondarenko, N.P. Linear differential operators with distribution coefficients of various singularity orders, Math. Meth. Appl. Sci. (2022), published online, DOI: http://doi.org/10.1002/mma.8929

\bibitem{Bond22-alg}
Bondarenko, N.P. Reconstruction of higher-order differential operators by their spectral data, Mathematics 10 (2022), no. 20, Article ID 3882 (32 pp.) 


\bibitem{MS16}
Mirzoev, K.A.; Shkalikov, A.A. Differential operators of even order with distribution coefficients, Math. Notes 99 (2016), no.~5, 779--784.

\bibitem{MS19}
Mirzoev, K.A.; Shkalikov, A.A. Ordinary differential operators of odd order with distribution coefficients, preprint (2019), arXiv:1912.03660 [math.CA].

\bibitem{Vlad17}
Vladimirov, A.A. On one approach to definition of singular differential operators, preprint (2017), arXiv:1701.08017 [math.SP].


\bibitem{Bond23-asympt}
Bondarenko, N.P. Spectral data asymptotics for the higher-order differential operators with distribution coefficients, J. Math. Sci. (2023), published online.
DOI: https://doi.org/10.1007/s10958-022-06118-x

\bibitem{SS20}
Savchuk, A.M.; Shkalikov, A.A. Asymptotic analysis of solutions of ordinary differential equations with distribution coefficients, Sb. Math. 211 (2020), no.~11, 1623--1659.


\bibitem{Nai68}
Naimark, M.A. Linear Differential Operators, 2nd ed., Nauka, Moscow (1969); English transl. of 1st ed., Parts I,II, Ungar, New York (1967, 1968).


\bibitem{But07}
Buterin, S.A. On inverse spectral problem for non-selfadjoint Sturm-Liouville operator on a finite interval, J. Math. Anal. Appl. 335 (2007), no. 1, 739--749.

\bibitem{BSY13}
Buterin, S.A.; Shieh, C.-T.; Yurko, V.A. Inverse spectral problems for non-selfadjoint second-order differential operators with Dirichlet boundary conditions, Boundary Value Problems (2013), 2013:180.


\bibitem{KZ96}
Kong, Q.; Zettl, A. Eigenvalues of regular Sturm-Liouville problems, J. Diff. Eqns. 131, no.~1, 1--19.


\bibitem{Ao22-4ord}
Zhang, H.-Y.; Ao, J.-J.; Bo, F.-Z. Eigenvalues of fourth-order boundary value problems with distributional potentials, AIMS Mathematics 7 (2022), no.~5, 7294--7317.

\bibitem{HM03}
Hryniv, R.O.; Mykytyuk, Y.V. Inverse spectral problems for Sturm-Liouville operators with singular potentials, Inverse Problems 19 (2003), no.~3, 665--684.
\end{thebibliography}
\end{document}